\documentclass[a4paper,
	11pt,
	pdftex,
	headings=normal,
	headsepline,
	footsepline,
	onecolumn,
	headinclude,
	footinclude,
	DIV14,
	abstracton
]{scrartcl}

\usepackage{
	graphicx,
	amssymb,
	amsmath,
	amsthm,
	xcolor,
	dsfont,
	authblk,
}

\usepackage[ngerman, english]{babel}

\usepackage[bf]{caption}
\usepackage{amssymb}

\usepackage{enumerate}
\usepackage{algorithm}
\usepackage{algorithmic}
\usepackage{subcaption}    
\usepackage{tabularx}
\usepackage[left=2.50cm, right=2.50cm, top=2.00cm, bottom=3.50cm]{geometry}

\newtheorem{theorem}{Theorem}[section]
\newtheorem{corollary}[theorem]{Corollary}
\newtheorem{lemma}[theorem]{Lemma}

\newtheorem{definition}[theorem]{Definition}
\newtheorem{remark}[theorem]{Remark}
\newtheorem{example}[theorem]{Example}

\DeclareMathOperator*{\argmin}{arg\,min} % argmin
\DeclareMathOperator*{\conv}{conv}
\newcommand{\R}{\mathbb{R}}
\newcommand{\N}{\mathbb{N}}

\begin{document}
	\title{An efficient descent method for locally Lipschitz multiobjective optimization problems}
	\author[1]{Bennet Gebken}
	\author[1]{Sebastian Peitz}
	\affil[1]{\normalsize Department of Mathematics, Paderborn University, Germany}
	
	\maketitle

	\begin{abstract}
		In this article, we present an efficient descent method for locally Lipschitz continuous multiobjective optimization problems (MOPs). The method is realized by combining a theoretical result regarding the computation of descent directions for nonsmooth MOPs with a practical method to approximate the subdifferentials of the objective functions. We show convergence to points which satisfy a necessary condition for Pareto optimality. Using a set of test problems, we compare our method to the multiobjective proximal bundle method by M{\"a}kel{\"a}. The results indicate that our method is competitive while being easier to implement. While the number of objective function evaluations is larger, the overall number of subgradient evaluations is lower. Finally, we show that our method can be combined with a subdivision algorithm to compute entire Pareto sets of nonsmooth MOPs.
	\end{abstract}
	
	\section{Introduction}
	In many scenarios in real life, the problem of optimizing multiple objectives at the same time arises. In engineering for example, one often wants to steer a physical system as close as possible to a desired state while minimizing the required energy cost at the same time. These problems are called \emph{multiobjective optimization problems} (MOPs) and generally do not possess a single optimal solution. Instead, the solution is the set of all optimal compromises, the so-called \emph{Pareto set} containing all \emph{Pareto optimal} points. Due to this, the numerical computation of solutions to MOPs is more challenging than to single-objective problems. On top of that, there are numerous applications where the objectives are nonsmooth, for example contact problems in mechanics, which adds to the difficulty. In this article, we will address both difficulties combined by considering nonsmooth MOPs.
	
	When addressing the above-mentioned difficulties, i.e., multiple objectives and nonsmoothness, separately, there exists a large number solution methods. For smooth MOPs, the most popular methods include evolutionary \cite{D2001, DPAM2002} and scalarization methods \cite{M1998}. Additionally, some methods from single-objective optimization have been generalized, like gradient descent methods \cite{FS2000,SSW2002,GPD2017} and Newton's method \cite{FGS2008}. For the nonsmooth single-objective case, commonly used methods include subgradient methods \cite{S1985}, bundle methods \cite{K1990} and gradient sampling methods \cite{BLO2005}. More recently, in \cite{MY2012}, a generalization of the steepest descent method to the nonsmooth case was proposed, which is based on an efficient approximation of the subdifferential of the objective function. For nonsmooth multiobjective optimization, the literature is a lot more scarce. Since classical scalarization approaches do not require the existence of gradients, they can still be used. In \cite{AGG2015}, a generalization of the steepest descent method was proposed for the case when the full subdifferentials of the objectives are known, which is rarely the case in practice. In \cite{B2013,NSFL2013}, the subgradient method was generalized to the multiobjective case, but both articles report that their method is not suitable for real life application due to inefficiency. In \cite{MKW2014} (see also \cite{K1985b,M2003}), a multiobjective version of the proximal bundle method was proposed, which currently appears to be the most efficient solver.
	
	In this article, we develop a new descent method for locally Lipschitz continuous MOPs by combining the descent direction from \cite{AGG2015} with the approximation of the subdifferentials from \cite{MY2012}. In \cite{AGG2015} it was shown that the element with the smallest norm in the negative convex hull of the subdifferentials of the objective functions is a common descent direction for all objectives. In \cite{MY2012}, the subdifferential of the objective function was approximated by starting with a single subgradient and then systematically computing new subgradients until the element with the smallest norm in the convex hull of all subgradients is a direction of (sufficient) descent. Combining both approaches yields a descent direction for locally Lipschitz MOPs and together with an Armijo step length, we obtain a descent method. We show convergence to points which satisfy a necessary condition for Pareto optimality. Using a set of test problems, we compare the performance of our method to the multiobjective proximal bundle method from \cite{MKW2014}. The results indicate that our method is inferior in terms of function evaluations, but superior in terms of subgradient evaluations.
	
	The structure of this article is as follows: we start with a short introduction to nonsmooth and multiobjective optimization in Section \ref{sec:introduction}. In Section \ref{sec:descent_method}, we derive our descent method by replacing the Clarke subdifferential for the computation of the descent direction by the Goldstein $\varepsilon$-subdifferential and then showing how the latter can be efficiently approximated. In Section \ref{sec:numerical_examples}, we apply our descent method to numerical examples. We first visualize and discuss the typical behavior of our method before comparing it to the multiobjective proximal bundle method from \cite{MKW2014} using a set of test problems. Afterwards, we show how our method can be combined with a subdivision algorithm to approximate entire Pareto sets. Finally, in Section \ref{sec:conclusion}, we draw a conclusion and discuss possible future work.
	
	\section{Nonsmooth multiobjective optimization} \label{sec:introduction}
	We consider the nonsmooth multiobjective optimization problem
	\begin{align} \tag{MOP} \label{eq:MOP}
		\min_{x \in \R^n} f(x) = \min_{x \in \R^n}
		\begin{pmatrix}
			f_1(x) \\
			\vdots \\
			f_k(x)
		\end{pmatrix},
	\end{align}
	where $f : \R^n \rightarrow \R^k$ is the \emph{objective vector} with components $f_i : \R^n \rightarrow \R$, $i \in \{1,...,k\}$, called \emph{objective functions}. We assume the objective functions to be \emph{locally Lipschitz continuous}, i.e., for each $i \in \{1,...,k\}$ and $x \in \R^n$, there is some $L_i > 0$ and $\varepsilon > 0$ with
	\begin{align*}
		|f_i(y)- f_i(z)| \leq L_i \| y - z \| \quad \forall y, z \in \{ y \in \R^n : \| x - y \| < \varepsilon \},
	\end{align*}		
	where $\| \cdot \|$ denotes the Euclidean norm in $\R^n$. Since \eqref{eq:MOP} is an optimization problem with a vector-valued objective function, the classical concept of optimality from the scalar case can not directly be conveyed. Instead, we are looking for the \emph{Pareto set}, which is defined in the following way:
		
	\begin{definition}
		A point $x \in \R^n$ is called \emph{Pareto optimal}, if there is no $y \in \R^n$ such that
		\begin{align*}
					f_i(y) &\leq f_i(x) \quad \forall i \in \{1,...,k\}, \\
					f_j(y) &< f_j(x) \quad \text{for some } j \in \{1,...,k\}.
		\end{align*}
		The set of all Pareto optimal points is the \emph{Pareto set}.					
	\end{definition}		

	In practice, to check if a given point is Pareto optimal, we need optimality conditions. In the smooth case, there are the well-known KKT conditions (cf.~\cite{M1998}), which are based on the gradients of the objective functions. In case the objective functions are merely locally Lipschitz, the KKT conditions can be generalized using the concept of \emph{subdifferentials}. In the following, we will recall the required definitions and results from nonsmooth analysis. For a more detailed introduction, we refer to \cite{C1983}.		

	\begin{definition}
		Let $\Omega_i \subseteq \R^n$ be the set of points where $f_i$ is not differentiable. Then
		\begin{align*}
			\partial f_i(x) = \conv ( \{ \xi \in \R^n : &\exists (x_j)_j \in \R^n \setminus \Omega_i \text{ with } x_j \rightarrow x \text{ and } \\
			&\nabla f_i(x_j) \rightarrow \xi \text{ for } j \rightarrow \infty \} )
		\end{align*}
		is the \emph{(Clarke) subdifferential of} $f_i$ \emph{in} $x$. $\xi \in \partial f_i(x)$ is a \emph{subgradient}.
	\end{definition}
	
	It is easy to see that if $f_i$ is continuously differentiable, then the Clarke subdifferential is the set containing only the gradient of $f_i$. We will later use the following technical result on some properties of the Clarke subdifferential (cf.~\cite{C1983}, Prop.~2.1.2).
	
	\begin{lemma}  \label{lem:subdiff}
		$\partial f_i(x)$ is nonempty, convex and compact.
	\end{lemma}
	
	Using the subdifferential, we can state a necessary optimality condition for locally Lipschitz MOPs (cf.~\cite{MEK2014}, Thm.~12).
	
	\begin{theorem} \label{thm:KKT}
		Let $x \in \R^n$ be Pareto optimal. Then
		\begin{align} \label{eq:KKT}
			0 \in \conv \left( \bigcup_{i=1}^k \partial f_i(x) \right).
		\end{align}
	\end{theorem}
	
	In the smooth case, \eqref{eq:KKT} reduces to the classical multiobjective KKT conditions. Note that in contrast to the smooth case, the optimality condition \eqref{eq:KKT} is numerically challenging to work with, as subdifferentials are difficult to compute. Thus, in numerical methods, \eqref{eq:KKT} is only used implicitly.
	
	The method we are presenting in this paper is a \emph{descent method}, which means that, starting from a point $x_1 \in \R^n$, we want to generate a sequence $(x_j)_j \in \R^n$ in which each point is an improvement over the previous point. This is done by computing directions $v_j \in \R^n$ and step lengths $t_j \in \R^{>0}$ such that $x_{j+1} = x_j + t_j v_j$ and
	\begin{align*}
		f_i(x_{j+1}) < f_i(x_j) \quad \forall j \in \N, \ i \in \{1,...,k\}.
 	\end{align*}
	For the computation of $v_j$, we recall the following basic result from convex analysis that forms the theoretical foundation for descent methods in the presence of multiple (sub)gradients. Let $\| . \|$ be the Euclidean norm in $\R^n$.
	
	\begin{theorem} \label{thm:steepest_descent_direction}
		Let $W \subseteq \R^n$ be convex and compact and 
		\begin{align} \label{eq:min_norm_problem}
			\bar{v} := \argmin_{\xi \in -W} \| \xi \|^2.
		\end{align}		
		Then either $\bar{v} \neq 0$ and
		\begin{align} \label{eq:steepest_descent_ineq}
			\langle \bar{v}, \xi \rangle \leq - \| \bar{v} \|^2 < 0 \quad \forall \xi \in W,
		\end{align}				
		or $\bar{v} = 0$ and there is no $v \in \R^n$ with $\langle v, \xi \rangle < 0$ for all $\xi \in W$.
	\end{theorem}
	\begin{proof}
		Since $\bar{v}$ is the projection of the origin onto the closed and convex set $-W$, we have
		\begin{align*}
								  & 0 \leq \langle -\xi - \bar{v}, \bar{v} - 0 \rangle = -\langle \bar{v}, \xi \rangle - \| \bar{v} \|^2 \\
			\Leftrightarrow \quad & \langle \bar{v}, \xi \rangle \leq - \| \bar{v} \|^2
		\end{align*}
		for all $\xi \in W$ (cf.~\cite{CG1959}, Lem.). In particular, if $\bar{v} \neq 0$ then $\langle \bar{v}, \xi \rangle \leq - \| \bar{v} \|^2 < 0$.
		
		Conversely, $\bar{v} = 0$ implies $0 \in W$, so in this case there can not be any $v \in \R^n$ with $\langle v, \xi \rangle < 0$ for all $\xi \in W$.
	\end{proof}	
	
	Roughly speaking, Theorem \ref{thm:steepest_descent_direction} states that the element of minimal norm in the convex and compact set $-W$ is directionally opposed to all elements of $W$. To be more precise, $\bar{v}$ is contained in the intersection of all half-spaces induced by elements of $-W$. In the context of optimization, this result has several applications:
	\begin{itemize}
		\item[(i)] In the smooth, single-objective case, $W = \{ \nabla f(x) \}$ trivially yields the classical steepest descent method.
		\item[(ii)] In the smooth, multiobjective case, $W = \conv(\{ \nabla f_1(x), ..., \nabla f_k(x) \})$ yields the descent direction from \cite{FS2000} (after dualization) and \cite{SSW2002}.
		\item[(iii)] In the nonsmooth, single-objective case, $W = \partial f(x)$ yields the descent direction from \cite{C1983}, Prop.~6.2.4.
		\item[(iv)] In the nonsmooth, multiobjective case, $W = \conv \left( \bigcup_{i = 1}^k \partial f_i(x) \right)$ yields the descent direction from \cite{AGG2015}. 
	\end{itemize}
	
	In (i) and (ii), the solution of problem \eqref{eq:min_norm_problem} is straightforward, since $W$ is a convex polytope with the gradients as vertices. In (iii), the solution of \eqref{eq:min_norm_problem} is non-trivial due to the difficulty of computing the subdifferential. In subgradient methods \cite{S1985}, the solution is approximated by using a single subgradient instead of the entire subdifferential. As a result, it can not be guaranteed that the solution is a descent direction and in particular, \eqref{eq:min_norm_problem} can not be used as a stopping criterion. In gradient sampling methods \cite{BLO2005}, the subdifferential is approximated by the convex hull of gradients of the objective function in randomly sampled points around the current point. Due to the randomness, it can not be guaranteed that the resulting direction yields sufficient descent. Additionally, a check for differentiability of the objective is required, which can pose a problem \cite{HSS2016}. In (iv), the solution of \eqref{eq:min_norm_problem} gets even more complicated due to the presence of multiple subdifferentials. So far, the only methods that deal with \eqref{eq:min_norm_problem} in this case are multiobjective versions of the subgradient method \cite{B2013,NSFL2013}, which were reported unsuitable for real life applications.
	
	In the following section, we will describe a new way to compute descent directions for nonsmooth MOPs by systematically computing an approximation of $\conv \left( \cup_{i = 1}^k \partial f_i(x) \right)$ that is sufficient to obtain a "good enough" descent direction from \eqref{eq:min_norm_problem}.
	
	\section{Descent method for nonsmooth MOPs}	\label{sec:descent_method}
	
	In this section, we will present a method to compute descent directions of nonsmooth MOPs that generalizes the method from \cite{MY2012} to the multiobjective case. As described in the previous section, when computing descent directions via Theorem \ref{thm:steepest_descent_direction}, one has the problem of having to compute subdifferentials. Since these are difficult to come by in practice, we will instead replace $W$ in Theorem \ref{thm:steepest_descent_direction} by an approximation of $\conv \left( \cup_{i = 1}^k \partial f_i(x) \right)$ such that the resulting direction is guaranteed to have sufficient descent. To this end, we will first replace the Clarke subdifferential by the so-called $\varepsilon$-\emph{subdifferential}, and then take a finite approximation of the latter.
	
	\subsection{The epsilon-subdifferential}
	By definition, $\partial f_i(x)$ is the convex hull of the limits of the gradient of $f_i$ in all sequences near $x$ that converge to $x$. Thus, if we evaluate $\nabla f_i$ in a number of points close to $x$ (where it is defined) and take the convex hull, we expect the resulting set to be an approximation of $\partial f_i(x)$. To formalize this, we introduce the following definition \cite{G1977,K2010}.
	
	\begin{definition}
		Let $\varepsilon \geq 0$, $x \in \R^n$ and $B_\varepsilon(x) := \{ y \in \R^n : \| x - y \| \leq \varepsilon\}$. Then
		\begin{align*}
			\partial_\varepsilon f_i(x) := \conv \left( \bigcup_{y \in B_\varepsilon(x)} \partial f_i(y) \right)
		\end{align*}
		is the \emph{(Goldstein)} $\varepsilon$-\emph{subdifferential of} $f_i$ \emph{in} $x$. $\xi \in \partial_\varepsilon f_i(x)$ is an $\varepsilon$-\emph{subgradient}.
	\end{definition}
	
	Note that $\partial_0 f_i(x) = \partial f_i(x)$ and $\partial f_i(x) \subseteq \partial_\varepsilon f_i(x)$. For $\varepsilon \geq 0$ we define for the multiobjective setting
	\begin{align*}
		F_\varepsilon(x) := \conv \left( \bigcup_{i=1}^k \partial_\varepsilon f_i(x) \right).
	\end{align*}	
	To be able to choose $W = F_\varepsilon(x)$ in Theorem \ref{thm:steepest_descent_direction}, we first need to establish some properties of $F_\varepsilon(x)$.
	
	\begin{lemma}
		$\partial_\varepsilon f_i(x)$ is nonempty, convex and compact. In particular, the same holds for $F_\varepsilon(x)$.
	\end{lemma}
	\begin{proof}
		For $\partial_\varepsilon f_i(x)$, this was shown in \cite{G1977}, Prop.~2.3. For $F_\varepsilon(x)$, it then follows directly from the definition.
	\end{proof}
		
	We immediately get the following corollary from Theorems \ref{thm:KKT} and \ref{thm:steepest_descent_direction}.

	\begin{corollary}
		Let $\varepsilon \geq 0$. 
		\begin{itemize}
			\item[a)] If $x$ is Pareto optimal, then
				\begin{align} \label{eq:eps_critical}
					0 \in F_\varepsilon(x).
				\end{align}
			\item[b)] Let $x \in \R^n$ and
				\begin{align} \label{eq:eps_min_norm_problem}
					\bar{v} := \argmin_{\xi \in -F_\varepsilon(x)} \| \xi \|^2.
				\end{align}		
				Then either $\bar{v} \neq 0$ and
				\begin{align} \label{eq:eps_descent_ineq}
					\langle \bar{v}, \xi \rangle \leq - \| \bar{v} \|^2 < 0 \quad \forall \xi \in F_\varepsilon(x),
				\end{align}				
				or $\bar{v} = 0$ and there is no $v \in \R^n$ with $\langle v, \xi \rangle < 0$ for all $\xi \in F_\varepsilon(x)$.
			\end{itemize}
	\end{corollary}	
	
	The previous corollary states that when working with the $\varepsilon$-subdifferential instead of the Clarke subdifferential, we still have a necessary optimality condition and a way to compute descent directions, although the optimality conditions are weaker and the descent direction has a less strong descent. This is illustrated in the following example.
	
	\begin{example} \label{exam:epssubdiff}
		Consider the locally Lipschitz function
		\begin{align*}
			f : \R^2 \rightarrow \R^2, \quad x \mapsto 
			\begin{pmatrix}
				(x_1 - 1)^2 + (x_2 - 1)^2 \\
				x_1^2 + | x_2 |
			\end{pmatrix}.
		\end{align*}
		The set of nondifferentiable points of $f$ is $\R \times \{ 0 \}$. For $\varepsilon > 0$ and $x \in \R^2$ we have
		\begin{align*}
			\nabla f_1(x) = 
			\begin{pmatrix}
				2(x_1 - 1) \\
				2(x_2 - 1)
			\end{pmatrix}
			\quad \text{and} \quad
			\partial_\varepsilon f_1(x) =  2 B_\varepsilon(x) - 
			\begin{pmatrix}
				2 \\
				2
			\end{pmatrix}.
		\end{align*}
		For $x \in \R \times \{ 0 \}$ we have
		\begin{align*}
			\partial f_2(x) = \{ 2 x_1 \} \times [-1,1] \quad \text{and} \quad \partial_\varepsilon f_2(x) = \{ 2 x_1 + [-2\varepsilon,2\varepsilon] \} \times [-1,1].
		\end{align*}
		
		Figure \ref{fig:example_epssubdiff_descent} shows the Clarke subdifferential (a), the $\varepsilon$-subdifferential (b) for $\varepsilon = 0.2$ and the corresponding sets $F_\varepsilon(x)$ for $x = (1.5, 0)^\top$. 
		\begin{figure}[ht] 
			\parbox[b]{0.49\textwidth}{
				\centering 
				\includegraphics[width=0.45\textwidth]{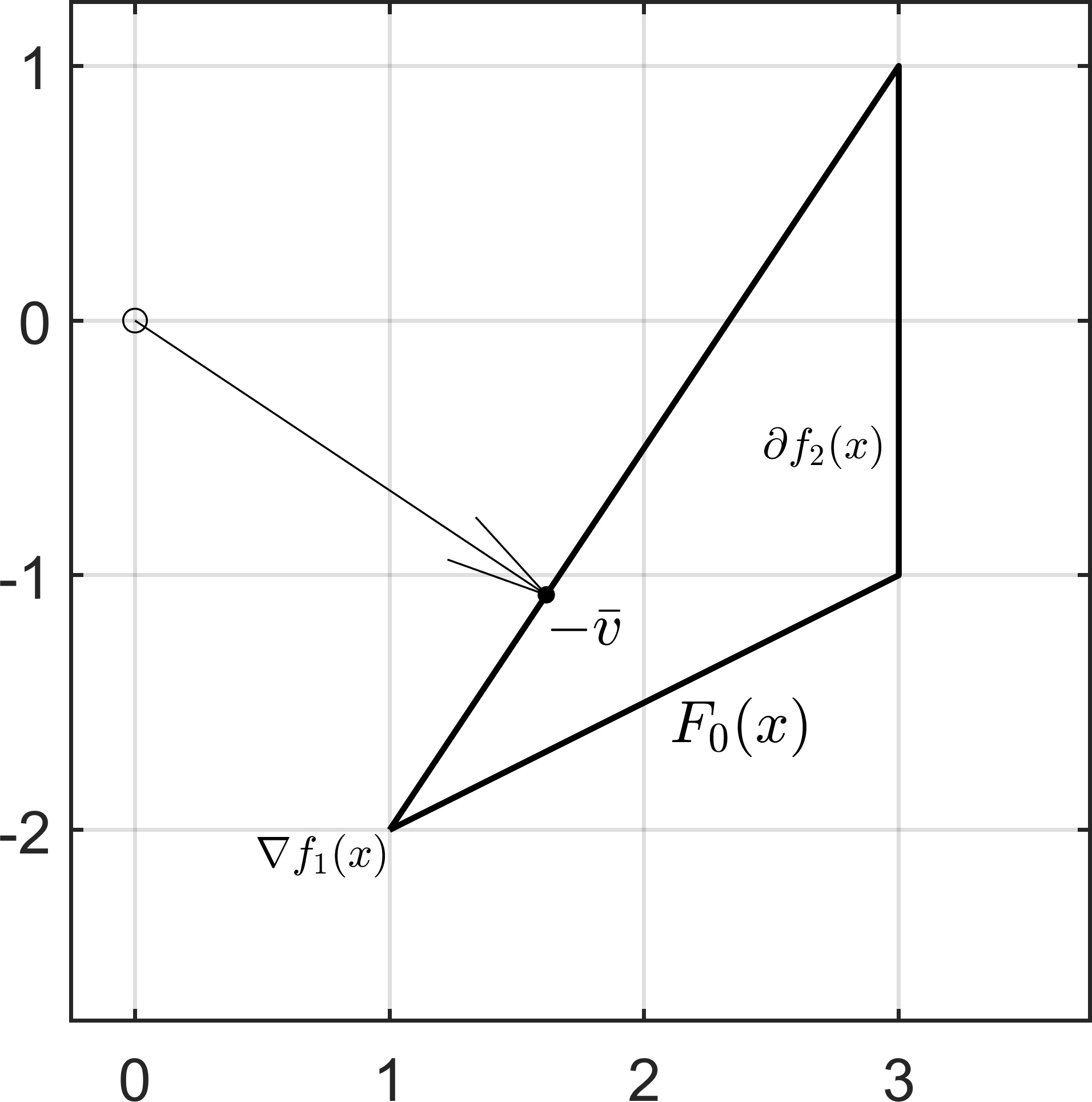}\\
				\textbf{(a)}
			}
			\parbox[b]{0.49\textwidth}{
				\centering 
				\includegraphics[width=0.45\textwidth]{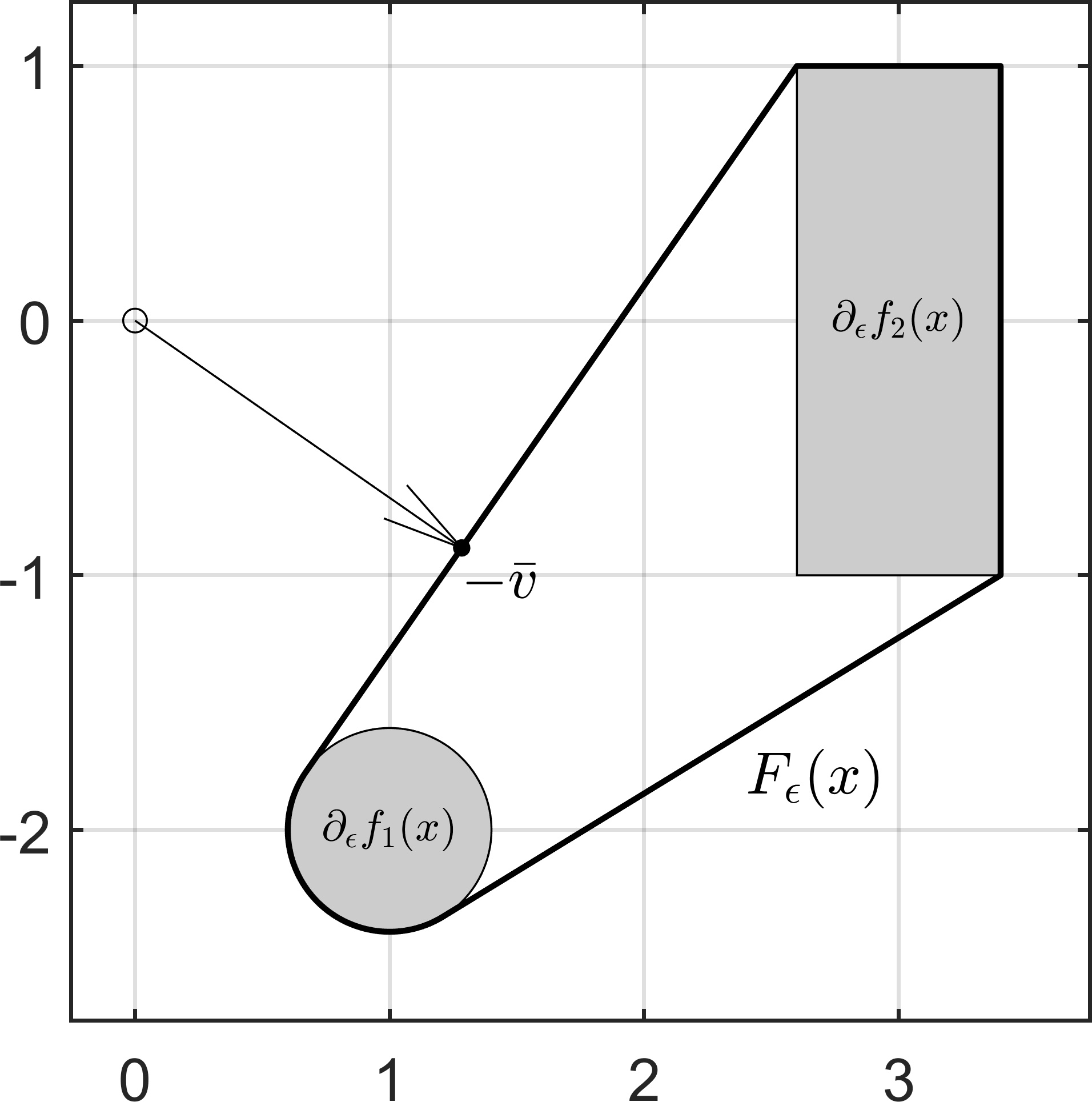}\\
				\textbf{(b)} 
			}
			\caption{Clarke subdifferentials (a), $\varepsilon$-subdifferentials (b) for $\varepsilon = 0.2$ and the corresponding sets $F_\varepsilon(x)$ for $x = (1.5, 0)^\top$ in Example \ref{exam:epssubdiff}.}
			\label{fig:example_epssubdiff_descent}
		\end{figure}
		Additionally, the corresponding solutions of \eqref{eq:eps_min_norm_problem} are shown. In this case, the predicted descent $-\| \bar{v} \|^2$ (cf.~\eqref{eq:steepest_descent_ineq}) is approximately $-3.7692$ in (a) and $-2.4433$ in (b). 
		
		Figure \ref{fig:example_epssubdiff_critical} shows the same scenario for $x = (0.5,0)^\top$.
		\begin{figure}[ht] 
			\parbox[b]{0.49\textwidth}{
				\centering 
				\includegraphics[width=0.45\textwidth]{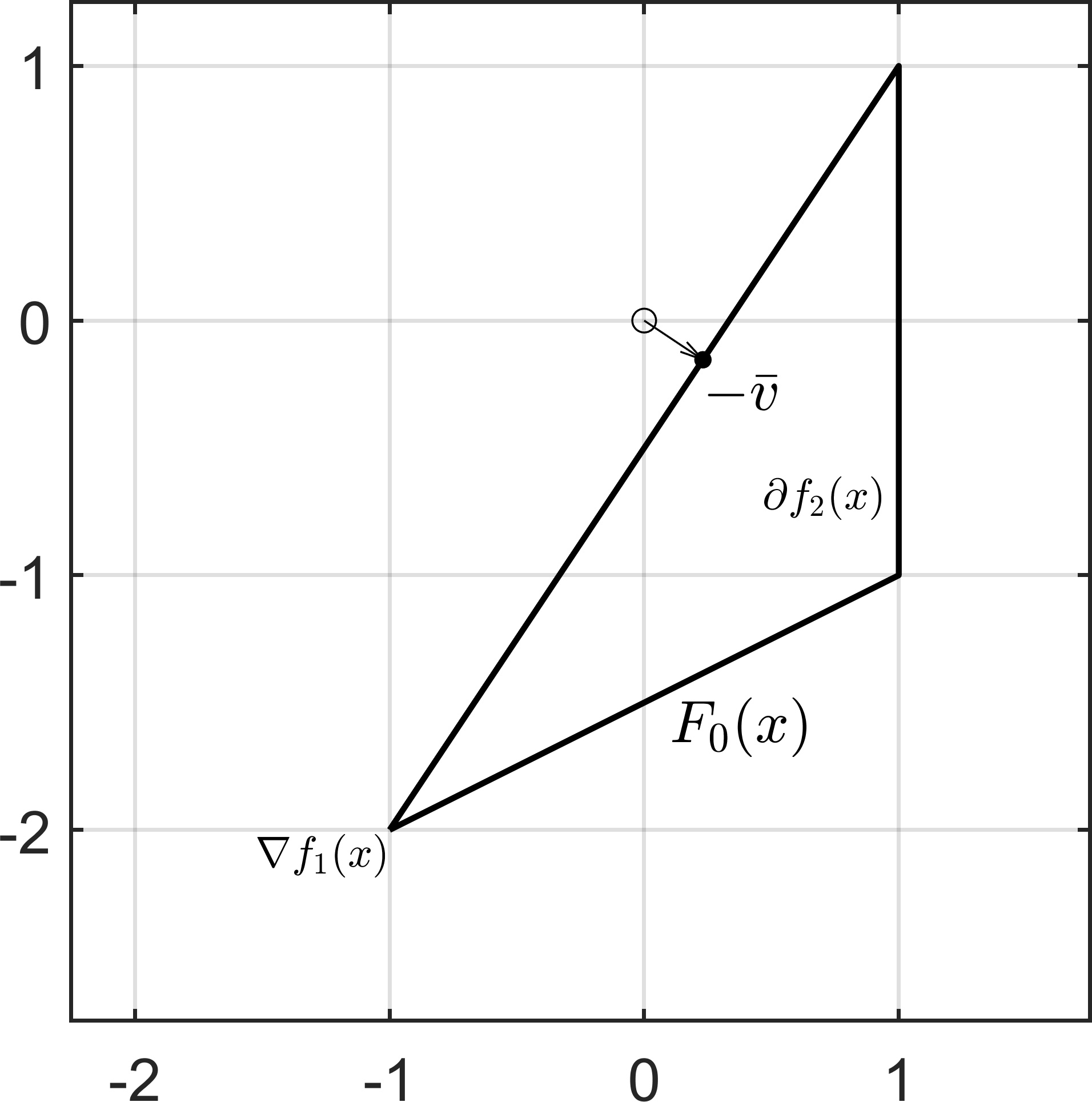}\\
				\textbf{(a)}
			}
			\parbox[b]{0.49\textwidth}{
				\centering 
				\includegraphics[width=0.45\textwidth]{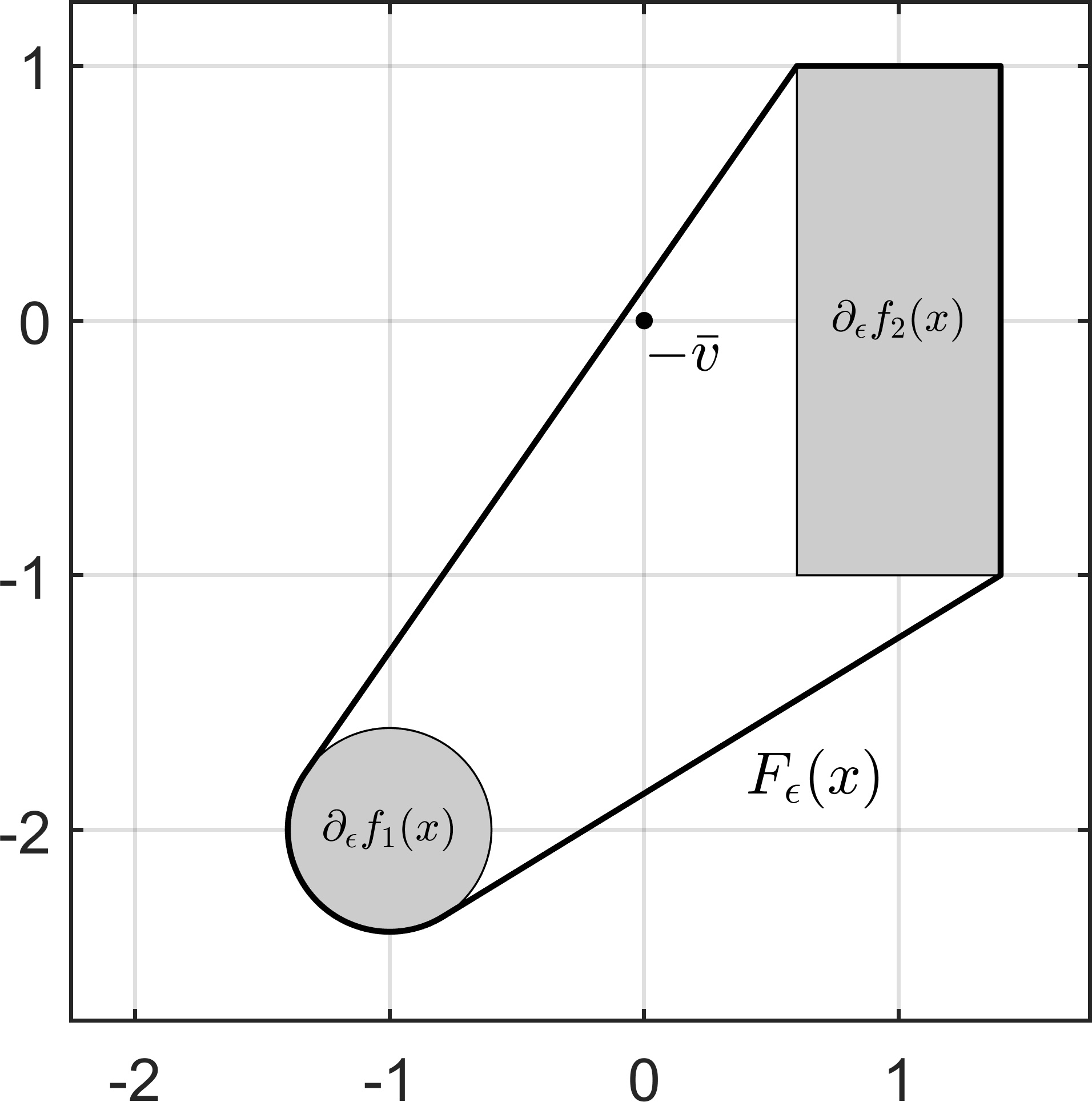}\\
				\textbf{(b)} 
			}
			\caption{Clarke subdifferentials (a), $\varepsilon$-subdifferentials (b) for $\varepsilon = 0.2$ and the corresponding sets $F_\varepsilon(x)$ for $x = (0.5, 0)^\top$ in Example \ref{exam:epssubdiff}.}
			\label{fig:example_epssubdiff_critical}
		\end{figure}
		Here, the Clarke subdifferential still yields a descent, while $\bar{v} = 0$ for the $\varepsilon$-subdifferential. In other words, $x$ satisfies the necessary optimality condition \eqref{eq:eps_critical} but not \eqref{eq:KKT}.
	\end{example}
	
	The following lemma shows that for the direction from \eqref{eq:eps_min_norm_problem}, there is a lower bound for a step size up to which we have guaranteed descent in each objective function $f_i$.
		
	\begin{lemma} \label{lem:step_size_bound}
		Let $\varepsilon \geq 0$ and $\bar{v}$ be the solution of \eqref{eq:eps_min_norm_problem}. Then
		\begin{align*}
			f_i(x + t \bar{v}) \leq f_i(x) - t \| \bar{v} \|^2 \quad \forall t \leq \frac{\varepsilon}{\| \bar{v} \|}.
		\end{align*}
	\end{lemma}
	\begin{proof}
		Let $t \leq \frac{\varepsilon}{\| \bar{v} \|}$. Since $f_i$ is locally Lipschitz continuous on $\R^n$, it is in particular Lipschitz continuous on an open set containing $x + [0,t] \bar{v}$. By applying the mean value theorem (cf.~\cite{C1983}, Thm.~2.3.7), we obtain
		\begin{align*}
			f_i(x + t \bar{v}) - f_i(x) \in \langle \partial f_i(x + r \bar{v}), t \bar{v} \rangle
		\end{align*}			
		for some $r \in (0,t)$. Since $\| x - (x + r \bar{v}) \| = r \| \bar{v} \| < \varepsilon$ we have $\partial f_i(x + r \bar{v}) \subseteq \partial_\varepsilon f_i(x)$. This means that there is some $\xi \in \partial_\varepsilon f_i(x)$ such that 
		\begin{align*}
			f_i(x + t \bar{v}) - f_i(x) = t \langle \xi, \bar{v} \rangle.
		\end{align*}			
		Combined with \eqref{eq:eps_descent_ineq} we obtain
		\begin{align*}
			& f_i(x + t \bar{v}) - f_i(x) \leq - t \| \bar{v} \|^2 \\
			\Leftrightarrow \quad & f_i(x + t \bar{v}) \leq f_i(x) - t \| \bar{v} \|^2.
		\end{align*}
	\end{proof}
	
	In the following, we will describe how we can obtain a good approximation of \eqref{eq:eps_min_norm_problem} without requiring full knowledge of the $\varepsilon$-subdifferentials. 	
	
	\subsection{Efficient computation of descent directions}	
	
	In this part, we will describe how the solution of \eqref{eq:eps_min_norm_problem} can be approximated when only a single subgradient can be computed at every $x \in \R^n$. Similar to the gradient sampling approach, the idea behind our method is to replace $F_\varepsilon(x)$ in \eqref{eq:eps_min_norm_problem} by the convex hull of a finite number of $\varepsilon$-subgradients $\xi_1,...,\xi_m \in F_\varepsilon(x)$, $m \in \N$. Since it is impossible to know a priori how many and which $\varepsilon$-subgradients are required to obtain a good descent direction, we solve \eqref{eq:eps_min_norm_problem} multiple times in an iterative approach to enrich our approximation until a satisfying direction has been found. To this end, we have to specify how to enrich our current approximation $\conv(\{ \xi_1, ..., \xi_m \})$ and how to characterize an acceptable descent direction.
	
	Let $W = \{\xi_1, ..., \xi_m\} \subseteq F_\varepsilon(x)$ and 
	\begin{align} \label{eq:approx_desc_dir}
		\tilde{v} := \argmin_{v \in -\conv(W)} \| v \|^2.
	\end{align}
	Let $c \in (0,1)$. Motivated by Lemma \ref{lem:step_size_bound}, we regard $\tilde{v}$ as an \emph{acceptable} descent direction, if
	\begin{align} \label{eq:accept_direction}
		f_i \left( x + \frac{\varepsilon}{\| \tilde{v} \|} \tilde{v} \right) \leq f_i(x) - c \varepsilon \| \tilde{v} \| \quad \forall i \in \{1,...,k\}.
	\end{align}	
	If the set $I \subseteq \{1,...,k\}$ for which \eqref{eq:accept_direction} is violated is non-empty	then we have to find a new $\varepsilon$-subgradient $\xi' \in F_\varepsilon(x)$ such that $W \cup \{ \xi' \}$ yields a better descent direction. Intuitively, \eqref{eq:accept_direction} being violated means that the local behavior of $f_i$, $i \in I$, in $x$ in the direction $\tilde{v}$ is not sufficiently captured in $W$. Thus, for each $i \in I$, we expect that there exists some $t' \in (0,\frac{\varepsilon}{\| \tilde{v} \|}]$ such that $\xi' \in \partial f_i(x + t' \tilde{v})$ improves the approximation of $F_\varepsilon(x)$. This is proven in the following lemma.
	
	\begin{lemma} \label{lem:infeasible_direction}
		Let $c \in (0,1)$, $W = \{ \xi_1, ..., \xi_m \} \subseteq F_\varepsilon(x)$ and $\tilde{v}$ be the solution of \eqref{eq:approx_desc_dir}. If
		\begin{align*}
			f_i \left( x + \frac{\varepsilon}{\| \tilde{v} \|} \tilde{v} \right) > f_i(x) - c \varepsilon \| \tilde{v} \|,
		\end{align*}
		then there is some $t' \in (0,\frac{\varepsilon}{\| \tilde{v} \|}]$ and $\xi' \in \partial f_i(x + t' \tilde{v})$ such that
		\begin{align} \label{eq:new_subgrad_condition}
			\langle \tilde{v}, \xi' \rangle > - c \| \tilde{v} \|^2.
		\end{align}
		In particular, $\xi' \in F_\varepsilon(x) \setminus \conv(W)$.
	\end{lemma}
	\begin{proof}
		Assume that for all $t' \in (0,\frac{\varepsilon}{\| \tilde{v} \|}]$ and $\xi' \in \partial f_i(x + t' \tilde{v})$ we have
		\begin{align} \label{eq:c_descent_contr}
			\langle \tilde{v}, \xi' \rangle \leq - c \| \tilde{v} \|^2.
		\end{align}
		By applying the mean value theorem as in Lemma \ref{lem:step_size_bound}, we obtain
		\begin{align*}
			f_i \left( x + \frac{\varepsilon}{\| \tilde{v} \|} \tilde{v} \right) - f_i(x) \in \langle \partial f_i(x + r \tilde{v}), \frac{\varepsilon}{\| \tilde{v} \|} \tilde{v} \rangle
		\end{align*}
		for some $r \in (0,\frac{\varepsilon}{\| \tilde{v} \|})$. This means that there is some $\xi' \in \partial f_i(x + r \tilde{v})$ such that
		\begin{align*}
			f_i \left( x + \frac{\varepsilon}{\| \tilde{v} \|} \tilde{v} \right) - f_i(x) = \langle \xi', \frac{\varepsilon}{\| \tilde{v} \|} \tilde{v} \rangle = \frac{\varepsilon}{\| \tilde{v} \|} \langle \xi',  \tilde{v} \rangle.
		\end{align*}
		By \eqref{eq:c_descent_contr} it follows that
		\begin{align*}
			& f_i \left(x + \frac{\varepsilon}{\| \tilde{v} \|} \tilde{v} \right) - f_i(x) \leq -c \varepsilon \| \tilde{v} \| \\
			\Leftrightarrow \quad & f_i \left(x + \frac{\varepsilon}{\| \tilde{v} \|} \tilde{v} \right) \leq f_i(x) - c \varepsilon \| \tilde{v} \|,
		\end{align*} 
		which is a contradiction. In particular, \eqref{eq:steepest_descent_ineq} yields $\xi' \in F_\varepsilon(x) \setminus \conv(W)$.
	\end{proof}	
	
	The following example visualizes the previous lemma.
	\begin{example} \label{exam:bad_descent}
		Consider $f$ as in Example \ref{exam:epssubdiff}, $\varepsilon = 0.2$ and $x = (0.75, 0)^\top$. The dashed lines in Figure \ref{fig:example_epssubdiff_false_descent} show the $\varepsilon$-subdifferentials, $F_\varepsilon(x)$ and the resulting descent direction (cf.~Figure \ref{fig:example_epssubdiff_descent} and \ref{fig:example_epssubdiff_critical}).
		\begin{figure}[ht] 
			\parbox[b]{0.49\textwidth}{
				\centering 
				\includegraphics[width=0.45\textwidth]{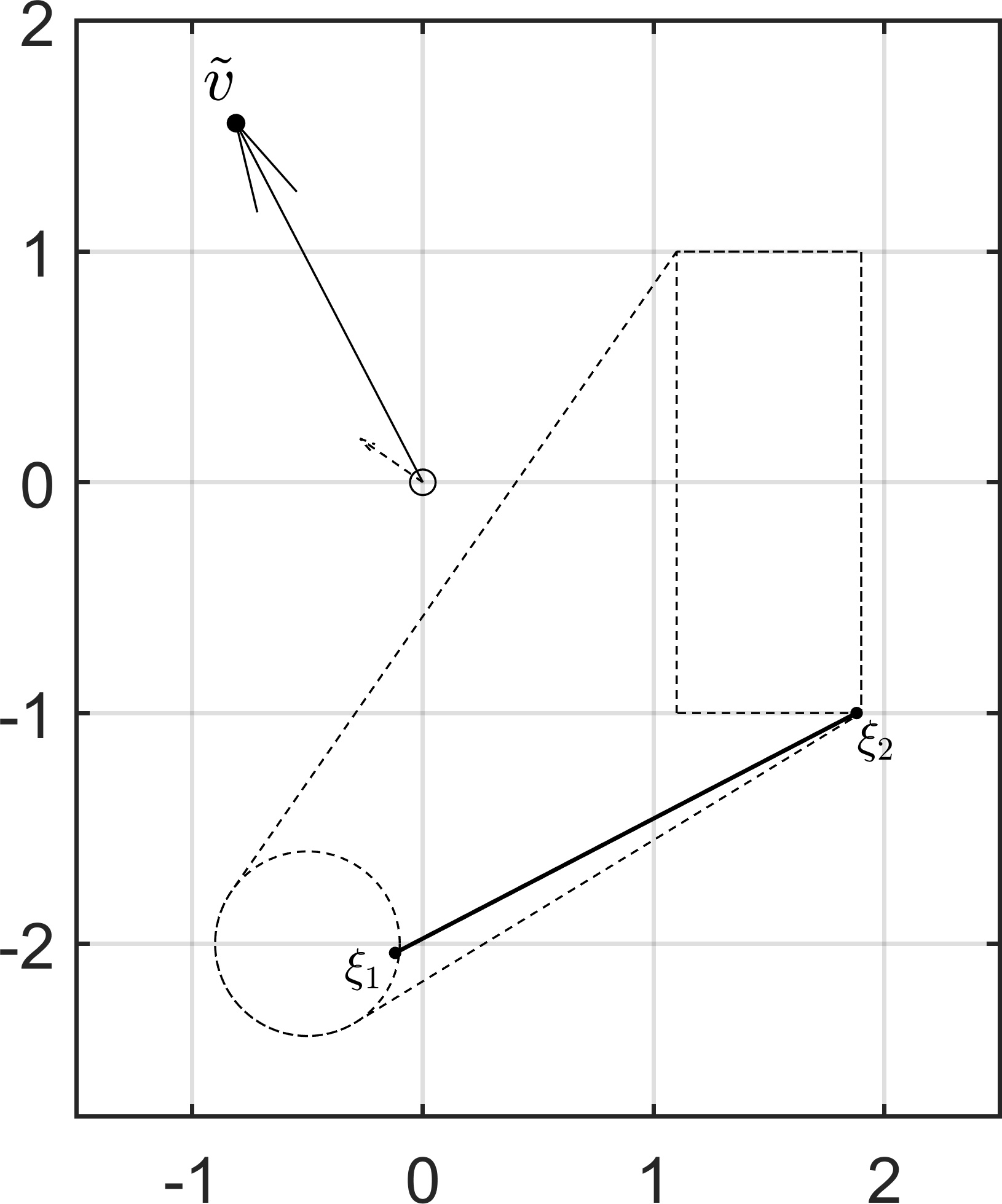}\\
				\textbf{(a)}
			}
			\parbox[b]{0.49\textwidth}{
				\centering 
				\includegraphics[width=0.45\textwidth]{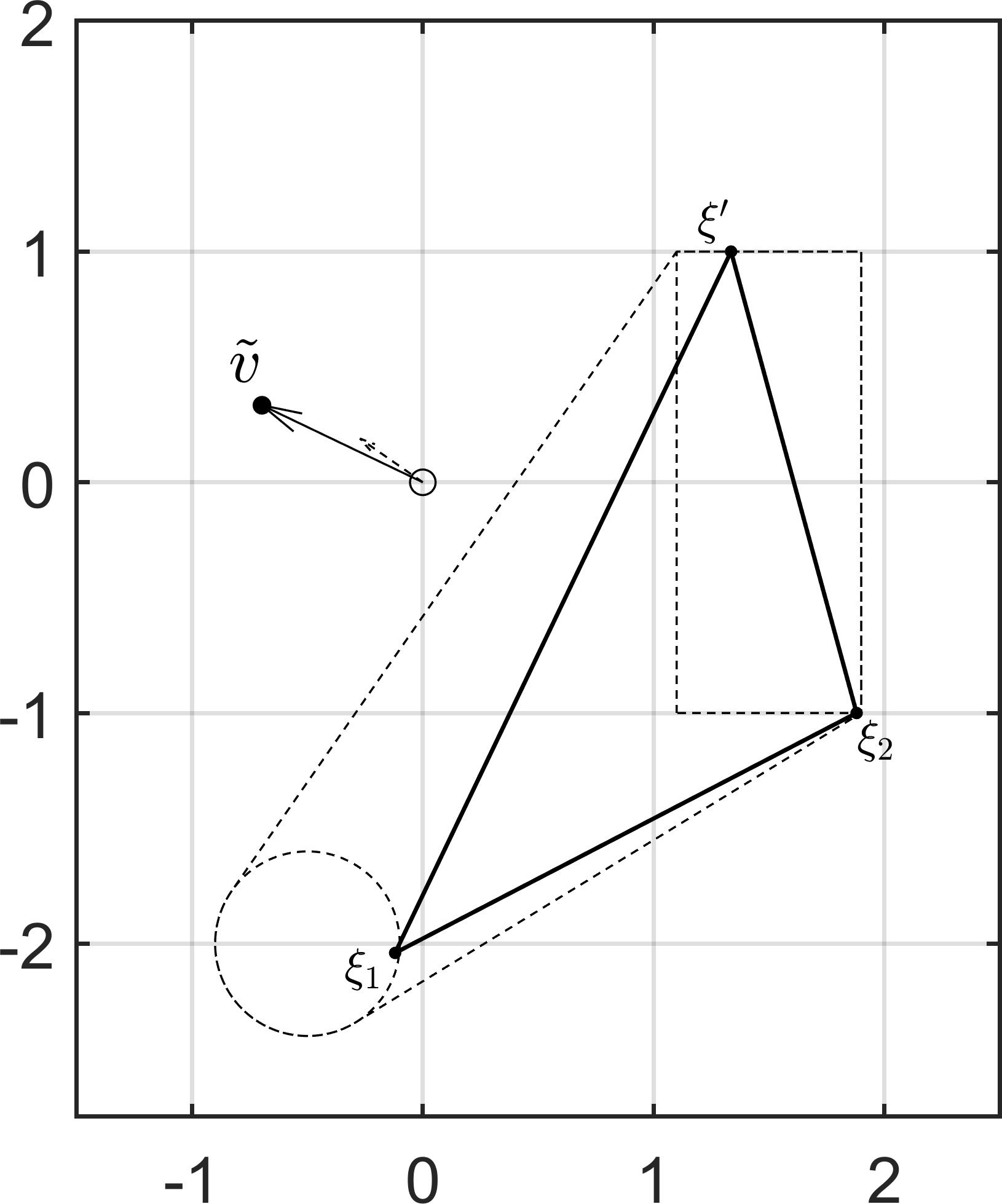}\\
				\textbf{(b)} 
			}
			\caption{Approximations of $F_\varepsilon(x)$ for $\varepsilon = 0.2$ and $x = (0.75, 0)^\top$ in Example \ref{exam:bad_descent}. $F_\epsilon(x)$ is approximated by $\conv(\{\xi_1, \xi_2\})$ in (a) and by $\conv(\{\xi_1, \xi_2, \xi' \})$ in (b).}
			\label{fig:example_epssubdiff_false_descent}
		\end{figure}
		Let $y = (0.94,-0.02)^\top$. Then $\| x - y \| \approx 0.191 \leq \varepsilon$, so $y \in B_\varepsilon(x)$ and
		\begin{align*}
			\partial_\varepsilon f_1(x) &\supseteq \partial f_1(y) = \left\{  
			\begin{pmatrix}
				-0.12 \\
				-2.04
			\end{pmatrix}						
			\right\}
			=: \{ \xi_1 \},\\
			\partial_\varepsilon f_2(x) &\supseteq \partial f_2(y) = \left\{  
			\begin{pmatrix}
				1.88 \\
				-1
			\end{pmatrix}						
			\right\}
			=: \{ \xi_2 \}.
		\end{align*}				
		Let $W := \{ \xi_1, \xi_2 \}$ and $\conv(W)$ be the approximation of $F_\varepsilon(x)$, shown as the solid line in Figure \ref{fig:example_epssubdiff_false_descent}(a). Let $\tilde{v}$ be the solution of \eqref{eq:approx_desc_dir} for this $W$ and choose $c = 0.25$. Checking \eqref{eq:accept_direction}, we have
		\begin{align*} 
			f_2 \left( x + \frac{\varepsilon}{\| \tilde{v} \|} \tilde{v} \right) &\approx 0.6101 > 0.4748 \approx f_2(x) - c \varepsilon \| \tilde{v} \|.
		\end{align*}
		By Lemma \ref{lem:infeasible_direction}, this means that there is some $t' \in (0, \frac{\varepsilon}{\| \tilde{v} \|}]$ and $\xi' \in \partial f_2(x + t' \tilde{v})$ such that
		\begin{align*}
			\langle \tilde{v}, \xi' \rangle > -c \| \tilde{v} \|^2.
		\end{align*}
		In this case, we can take for example $t' = \frac{1}{2} \frac{\varepsilon}{\| \tilde{v} \|}$, resulting in
		\begin{align*}
			\partial f_2(x + t'v) \approx \left\{
			\begin{pmatrix}
				1.4077 \\
				1
			\end{pmatrix}
			\right\} =: \{ \xi' \}, \\
			\langle \tilde{v}, \xi' \rangle \approx 0.4172 > -0.7696 \approx -c \| \tilde{v} \|^2.
		\end{align*}
		Figure \ref{fig:example_epssubdiff_false_descent}(b) shows the improved approximation $W \cup \{ \xi' \}$ and the resulting descent direction $\tilde{v}$. By checking \eqref{eq:accept_direction} for this new descent direction, we see that $\tilde{v}$ is acceptable. (Note that in general, a single improvement step like this will not be sufficient to reach an acceptable direction.)
	\end{example}	
	
	Note that Lemma \ref{lem:infeasible_direction} only shows the existence of $t'$ and $\xi'$ without stating a way how to actually compute them. To this end, let $i$ be the index of an objective function for which \eqref{eq:accept_direction} is not satisfied, define
	\begin{align} \label{eq:def_h}
		h_i : \R \rightarrow \R, \quad t \mapsto f_i(x + t \tilde{v}) - f_i(x) + c t \| \tilde{v} \|^2
	\end{align}
	(cf.~\cite{MY2012}) and consider Algorithm \ref{algo:new_subgradient}. If $f_i$ is continuously differentiable around $x$, then \eqref{eq:new_subgrad_condition} is equivalent to $h_i'(t') > 0$, i.e., $h_i$ being monotonically increasing around $t'$. Thus, the idea of Algorithm \ref{algo:new_subgradient} is to find some $t$ such that $h_i$ is monotonically increasing around $t$, while checking if \eqref{eq:new_subgrad_condition} is satisfied for a subgradient $\xi \in f_i(x + t \tilde{v})$.
	
	\begin{algorithm} 
		\caption{Compute new subgradient}
		\label{algo:new_subgradient}
		\begin{algorithmic}[1] 
			\REQUIRE Current point $x \in \R^n$, direction $\tilde{v}$, tolerance $\varepsilon$, Armijo parameter $c \in (0,1)$.
			\STATE Set $a = 0$, $b = \frac{\varepsilon}{\| \tilde{v} \|}$ and $t = \frac{a+b}{2}$.
			\STATE Compute $\xi' \in \partial f_i(x + t \tilde{v})$.
			\STATE If $\langle \tilde{v}, \xi' \rangle > - c \| \tilde{v} \|^2$ then stop.
			\STATE If $h_i(b) > h_i(t)$ then set $a = t$. Otherwise set $b = t$.
			\STATE Set $t = \frac{a+b}{2}$ and go to step 2.
		\end{algorithmic} 
	\end{algorithm}		
	
	Although in the general setting, we can not guarantee that Algorithm \ref{algo:new_subgradient} yields a subgradient satisfying \eqref{eq:new_subgrad_condition}, we can at least show that after finitely many iterations, a factor $t$ is found such that $\partial f_i(x + t \tilde{v})$ contains a subgradient that satisfies \eqref{eq:new_subgrad_condition}.
		
	\begin{lemma} \label{lem:conv_new_subgradient}
		Let $(t_k)_k$ be the sequence generated in Algorithm \ref{algo:new_subgradient}. If $(t_k)_k$ is finite, then some $\xi'$ was found such that \eqref{eq:new_subgrad_condition} is satisfied. If $(t_k)_k$ is infinite, then it converges to some $\bar{t} \in [0,\frac{\varepsilon}{\| \tilde{v} \|}]$ such that there is some $\xi' \in \partial f_i(x + \bar{t} \tilde{v})$ which satisfies \eqref{eq:new_subgrad_condition}. Additionally, there is some $N \in \N$ such that for all $k > N$ there is some $\xi' \in \partial f_i(x + t_k \tilde{v})$ satisfying \eqref{eq:new_subgrad_condition}.
	\end{lemma}
	\begin{proof}
		Let $(t_k)_k$ be finite with last element $\bar{t} \in (0,\frac{\varepsilon}{\| \tilde{v} \|})$. Then Algorithm \ref{algo:new_subgradient} must have stopped in step 3, i.e., some $\xi' \in \partial f_i(x + \bar{t} \tilde{v})$ satisfying \eqref{eq:new_subgrad_condition} was found. \\
		Now let $(t_k)_k$ be infinite. By construction, $(t_k)_k$ is a Cauchy sequence in the compact set $[0,\frac{\varepsilon}{\| \tilde{v} \|}]$, so it has to converge to some $\bar{t} \in [0,\frac{\varepsilon}{\| \tilde{v} \|}]$. Additionally, since \eqref{eq:accept_direction} is violated for the index $i$ by assumption, we have
		\begin{align*}
			h_i(0) = 0 \quad \text{and} \quad h_i \left( \frac{\varepsilon}{\| \tilde{v} \|} \right) > 0.
		\end{align*}				
		 Let $(a_k)_k$ and $(b_k)_k$ be the sequences corresponding to $a$ and $b$ in Algorithm \ref{algo:new_subgradient} (at the start of each iteration). Then $h_i(a_k) < h_i(b_k)$ for all $k \in \N$. Thus, by the mean value theorem, there has to be some $r_k \in (a_k,b_k)$ such that
		\begin{align*}
			0 < h_i(b_k) - h_i(a_k) \in \langle \partial h_i(r_k), b_k - a_k \rangle = \partial h_i(r_k) (b_k - a_k).
		\end{align*}
		In particular, $\lim_{k \rightarrow \infty} r_k = \bar{t}$ and since $a_k < b_k$, $\partial h_i(r_k) \cap \R^{> 0} \neq \emptyset$ for all $k \in \N$. By upper semicontinuity of $\partial h$ there must be some $\theta \in \partial h_i(\bar{t})$ with $\theta > 0$. By the chain rule, we have
		\begin{align} \label{eq:h_chain_rule}
			0 < \theta \in \partial h_i(\bar{t}) \subseteq \langle \tilde{v}, \partial f_i(x + \bar{t} \tilde{v}) \rangle + c \| \tilde{v} \|^2.
		\end{align}
		Thus, there must be some $\xi' \in \partial f_i(x + \bar{t} \tilde{v})$ with
		\begin{align*}
			& 0 < \langle \tilde{v}, \xi' \rangle + c \| \tilde{v} \|^2 \\
			\Leftrightarrow \quad & \langle \tilde{v}, \xi' \rangle > - c \| \tilde{v} \|^2.
		\end{align*}
		By upper semicontinuity of $\partial h$ it also follows that there is some $N \in \N$ such that $\partial h_i(t_k) \cap \R^{> 0} \neq \emptyset$ for all $k > N$. Applying the same argument as above completes the proof.
	\end{proof}
		
	In the following remark, we will briefly discuss the implication of Lemma \ref{lem:conv_new_subgradient} for practical use of Algorithm \ref{algo:new_subgradient}.
	
	\begin{remark}
		Let $N \in \N$ be as in Lemma \ref{lem:conv_new_subgradient}.
		\begin{itemize}
			\item[a)] Note that if $k > N$ and $h$ is differentiable in $t_k$, then we have
			\begin{align*}
				0 < \nabla h_i(t_k) = \langle \tilde{v}, \nabla f_i(x + t_k \tilde{v}) \rangle + c \| \tilde{v} \|^2,
			\end{align*}							
			i.e., the stopping criterion in step 3 is satisfied. Thus, if Algorithm \ref{algo:new_subgradient} generates an infinite sequence, $h$ must be nonsmooth in $t_k$ for all $k > N$. In particular, $f_i$ must be nonsmooth in $x + t_k \tilde{v}$ for all $k > N$.
			\item[b)] If $f$ is regular (cf.~\cite{C1983}, Def.~2.3.4), then equality holds when applying the chain rule to $h$ (cf.~\cite{C1983}, Thm. 2.3.10), i.e., 
			\begin{align*}
				\partial h_i(t_k) = \langle \tilde{v}, \partial f_i(x + t_k \tilde{v}) \rangle + c \| \tilde{v} \|^2.
			\end{align*}
			Thus, if Algorithm \ref{algo:new_subgradient} generates an infinite sequence, then for all $k > N$ there must be both positive and negative elements in $\partial h_i(t_k)$. By convexity of $\partial h_i(t_k)$, this implies that $0 \in \partial h_i(t_k)$ for all $k > N$, i.e., $h$ must have infinitely many (nonsmooth) stationary points in $[0,\frac{\varepsilon}{\| \tilde{v} \|}]$.
		\end{itemize}
	\end{remark}		
		
	Motivated by the previous remark, we will from now on assume that Algorithm \ref{algo:new_subgradient} stops after finitely many iterations and thus yields a new subgradient satisfying \eqref{eq:new_subgrad_condition}. We can use this method of finding new subgradients to construct an algorithm that computes descent directions of nonsmooth MOPs, namely Algorithm \ref{algo:descent_direction}.
	
	\begin{algorithm} 
		\caption{Compute descent direction}
		\label{algo:descent_direction}
		\begin{algorithmic}[1] 
			\REQUIRE Current point $x \in \R^n$, tolerances $\varepsilon, \delta > 0$, Armijo parameter $c \in (0,1)$.
			\STATE Compute $\xi^i_1 \in \partial_\varepsilon f_i(x)$ for all $i \in \{1,...,k\}$. Set $W_1 = \{ \xi^1_1, ..., \xi^k_1 \}$ and $l = 1$.
			\STATE Compute $v_l = \argmin_{v \in -\conv(W_l)} \| v \|^2$.
			\STATE If $\| v_l \| \leq \delta$ then stop.
			\STATE Find all objective functions for which there is insufficient descent:
				\begin{align*}
					I_l = \left\{ j \in \{1,...,k\} : f_j \left( x + \frac{\varepsilon}{\| v_l \|} v_l \right) > f_j(x) - c \varepsilon \| v_l \| \right\}.
				\end{align*}								
				If $I_l = \emptyset$ then stop.
			\STATE For each $j \in I_l$, compute $t \in (0,\frac{\varepsilon}{\| v_l \|}]$ and $\xi^j_l \in \partial f_j(x + t v_l)$ such that 
				\begin{align*}
					\langle v_l, \xi^j_l \rangle > -c \| v_l \|^2
				\end{align*}
				via Algorithm \ref{algo:new_subgradient}.
			\STATE Set $W_{l+1} = W_l \cup \{ \xi^j_l : j \in I_l \}$, $l = l+1$ and go to step 2.
		\end{algorithmic} 
	\end{algorithm}	
	
	The following theorem shows that Algorithm \ref{algo:descent_direction} stops after a finite number of iterations and produces an acceptable descent direction (cf.~\eqref{eq:accept_direction}).
		
	\begin{theorem} \label{thm:descent_dir_conv}
		Algorithm \ref{algo:descent_direction} terminates. In particular, if $\tilde{v}$ is the last element of $(v_l)_l$, then either $\| \tilde{v} \| \leq \delta$ or $\tilde{v}$ is an acceptable descent direction, i.e.,
		\begin{align*}
			f_i \left( x + \frac{\varepsilon}{\| \tilde{v} \|} \tilde{v} \right) \leq f_i(x) - c \varepsilon \| \tilde{v} \| \quad \forall i \in \{1,...,k\}.
		\end{align*}
	\end{theorem}			
	\begin{proof}
		Assume that Algorithm \ref{algo:descent_direction} does not terminate, i.e., $(v_l)_{l \in \N}$ is an infinite sequence. Let $l > 1$ and $j \in I_{l-1}$. Since $\xi^j_{l-1} \in W_l$ and $-v_{l-1} \in W_{l-1} \subseteq W_l$ we have
		\begin{align} \label{eq:proof_algo_est}
			\| v_l \|^2 &\leq \| -v_{l-1} + s (\xi^j_{l-1} + v_{l-1}) \|^2 \notag \\
			&= \| v_{l-1} \|^2 - 2 s \langle v_{l-1}, \xi^j_{l-1} + v_{l-1} \rangle + s^2 \| \xi^j_{l-1} + v_{l-1} \|^2 \notag \\
			&= \| v_{l-1} \|^2 - 2 s \langle v_{l-1}, \xi^j_{l-1} \rangle - 2 s \| v_{l-1} \|^2 + s^2 \| \xi^j_{l-1} + v_{l-1} \|^2 
		\end{align}
		for all $s \in [0,1]$. Since $j \in I_{l-1}$ we must have
		\begin{align} \label{eq:est_1}
			\langle v_{l-1}, \xi^j_{l-1} \rangle > -c \| v_{l-1} \|^2
		\end{align}				
		by step 5. Let $L$ be a common Lipschitz constant of all $f_i$, $i \in \{1,...,k\}$, on the closed $\varepsilon$-ball $B_\varepsilon(x)$ around $x$. Then by \cite{C1983}, Prop.~2.1.2, and the definition of the $\varepsilon$-subdifferential, we must have $\| \xi \| \leq L$ for all $\xi \in F_\varepsilon(x)$. So in particular,
		\begin{align} \label{eq:est_2}
			\| \xi^j_{l-1} + v_{l-1} \| \leq 2 L.
		\end{align}
		Combining \eqref{eq:proof_algo_est} with \eqref{eq:est_1} and \eqref{eq:est_2} yields
		\begin{align*}
			\| v_l \|^2 &< \| v_{l-1} \|^2 - 2 s c \| v_{l-1} \|^2 - 2 s \| v_{l-1} \|^2 + 4 s^2 L^2 \\
			&= \| v_{l-1} \|^2 - 2 s (c+1) \| v_{l-1} \|^2 + 4 s^2 L^2.
		\end{align*}
		Let $s := \frac{c+1}{4 L^2} \| v_{l-1} \|^2$. Since $c + 1 \in (1,2)$ and $\| v_{l-1} \| \leq L$ we have $s \in (0,1)$. We obtain
		\begin{align*}
			\| v_l \|^2 &< \| v_{l-1} \|^2 - 2 \frac{(c+1)^2}{4 L^2} \| v_{l-1} \|^4 + \frac{(c+1)^2}{4 L^2} \| v_{l-1} \|^4 \\
			&= \left( 1 - \frac{(c+1)^2}{4 L^2} \| v_{l-1} \|^2 \right) \| v_{l-1} \|^2.
		\end{align*}
		Since Algorithm \ref{algo:descent_direction} did not terminate, it holds $\| v_{l-1} \| > \delta$. It follows that
		\begin{align*}
			\| v_l \|^2 < \left( 1 - \left( \frac{c+1}{2 L} \delta \right)^2 \right) \| v_{l-1} \|^2.
		\end{align*}
		Let $r = 1 - \left( \frac{c+1}{2 L} \delta \right)^2$. Note that we have $\delta < \| v_l \| \leq L$ for all $l \in \N$, so $r \in (0,1)$. Additionally, $r$ does not depend on $l$, so we have
		\begin{align*}
			\| v_l \|^2 < r \| v_{l-1} \|^2 < r^2 \| v_{l-1} \|^2 < ... < r^{l-1} \| v_1 \|^2 \leq r^{l-1} L^2.
		\end{align*}
		In particular, there is some $l$ such that $\| v_l \| \leq \delta$, which is a contradiction.
	\end{proof}
	
	\begin{remark}
		The proof of Theorem \ref{thm:descent_dir_conv} shows that for convergence of Algorithm \ref{algo:descent_direction}, it would be sufficient to consider only a single $j \in I_j$ in step 5. Similarly, for the initial approximation $W_1$, a single element from $\partial_\varepsilon f_i(x)$ for any $i \in \{1,...,k\}$ would be enough. A modification of either step can potentially reduce the number of executions of step 5 (i.e., Algorithm \ref{algo:new_subgradient}) in Algorithm \ref{algo:descent_direction} in case the $\varepsilon$-subdifferentials of multiple objective functions are similar. Nonetheless, we will restrain the discussion in this article to Algorithm \ref{algo:descent_direction} as it is, since both modifications also introduce a bias towards certain objective functions, which we want to avoid.
	\end{remark}
	
	To highlight the strengths of Algorithm \ref{algo:descent_direction}, we will consider an example where standard gradient sampling approaches can fail to obtain a useful descent direction.
	
	\begin{example} \label{exam:compl_subdiff}
		For $a, b \in \R \setminus \{ 0 \}$ consider the locally Lipschitz function
		\begin{align*}
			f : \R^2 \rightarrow \R^2, \quad x \mapsto 
			\begin{pmatrix}
				(x_1 - 1)^2 + (x_2 - 1)^2 \\
				| x_2 - a |x_1|| + b x_2
			\end{pmatrix}.
		\end{align*}
		The set of nondifferentiable points is 
		\begin{align*}
			\Omega_f = (\{ 0 \} \times \R) \cup \{ (\lambda, a | \lambda| )^\top : \lambda \in \R \},
		\end{align*}
		separating $\R^2$ into four smooth areas (cf.~Figure \ref{fig:example_compl_subdiff}(a)). For large $a > 0$, the two areas above the graph of $\lambda \mapsto a | \lambda |$ become small, making it difficult to compute the subdifferential close to $(0,0)^\top$. 
		
		Let $a = 10$, $b = 0.5$, $\varepsilon = 10^{-3}$ and $x = (10^{-4}, 10^{-4})^\top$. In this case, $(0,0)^\top$ is the minimal point of $f_2$ and
		\begin{align*}
			\partial_\varepsilon f_2(x) &= \conv\left\{  
			\begin{pmatrix}
				-a \\
				b - 1
			\end{pmatrix},
			\begin{pmatrix}
				a \\
				b + 1
			\end{pmatrix},
			\begin{pmatrix}
				a \\
				b - 1
			\end{pmatrix},
			\begin{pmatrix}
				-a \\
				b + 1
			\end{pmatrix}					
			\right\} \\
			&= \conv\left\{  
			\begin{pmatrix}
				-10 \\
				-0.5
			\end{pmatrix},
			\begin{pmatrix}
				10 \\
				1.5
			\end{pmatrix},
			\begin{pmatrix}
				10 \\
				-0.5
			\end{pmatrix},
			\begin{pmatrix}
				-10 \\
				1.5
			\end{pmatrix}					
			\right\}.
		\end{align*}	
		In particular, $0 \in \partial_\varepsilon f_2(x)$, so the descent direction with the exact $\varepsilon$-subdifferentials from \eqref{eq:eps_min_norm_problem} is zero. When applying Algorithm \ref{algo:descent_direction} in $x$, after two iterations we obtain 
		\begin{align*}
			\tilde{v} = v_2 \approx (0.118, 1.185) \cdot 10^{-9},
		\end{align*}				
		i.e., $\| \tilde{v} \| \approx 1.191 \cdot 10^{-11}$. Thus, $x$ is correctly identified as (potentially) Pareto optimal.  The final approximation $W_2$ of $F_\varepsilon(x)$ is
		\begin{align*}
			W_2 = \left\{ \xi^1_1, \xi^2_1, \xi^2_2 \right\} = \left\{  
			\begin{pmatrix}
				10 \\
				-0.5
			\end{pmatrix},
			\begin{pmatrix}
				-1.9998 \\
				-1.9998
			\end{pmatrix},
			\begin{pmatrix}
				-10 \\
				1.5
			\end{pmatrix}
			\right\}.
		\end{align*}
		The first two elements of $W_2$ are the gradients of $f_1$ and $f_2$ in $x$ from the first iteration of Algorithm \ref{algo:descent_direction}, and the last element is the gradient of $f_2$ in $x' = x + tv_1 = (0.038, 0.596)^\top \cdot 10^{-3} \in B_\varepsilon(x)$ from the second iteration. The result is visualized in Figure \ref{fig:example_compl_subdiff}.
		\begin{figure}[ht] 
			\parbox[b]{0.49\textwidth}{
				\centering 
				\includegraphics[width=0.45\textwidth]{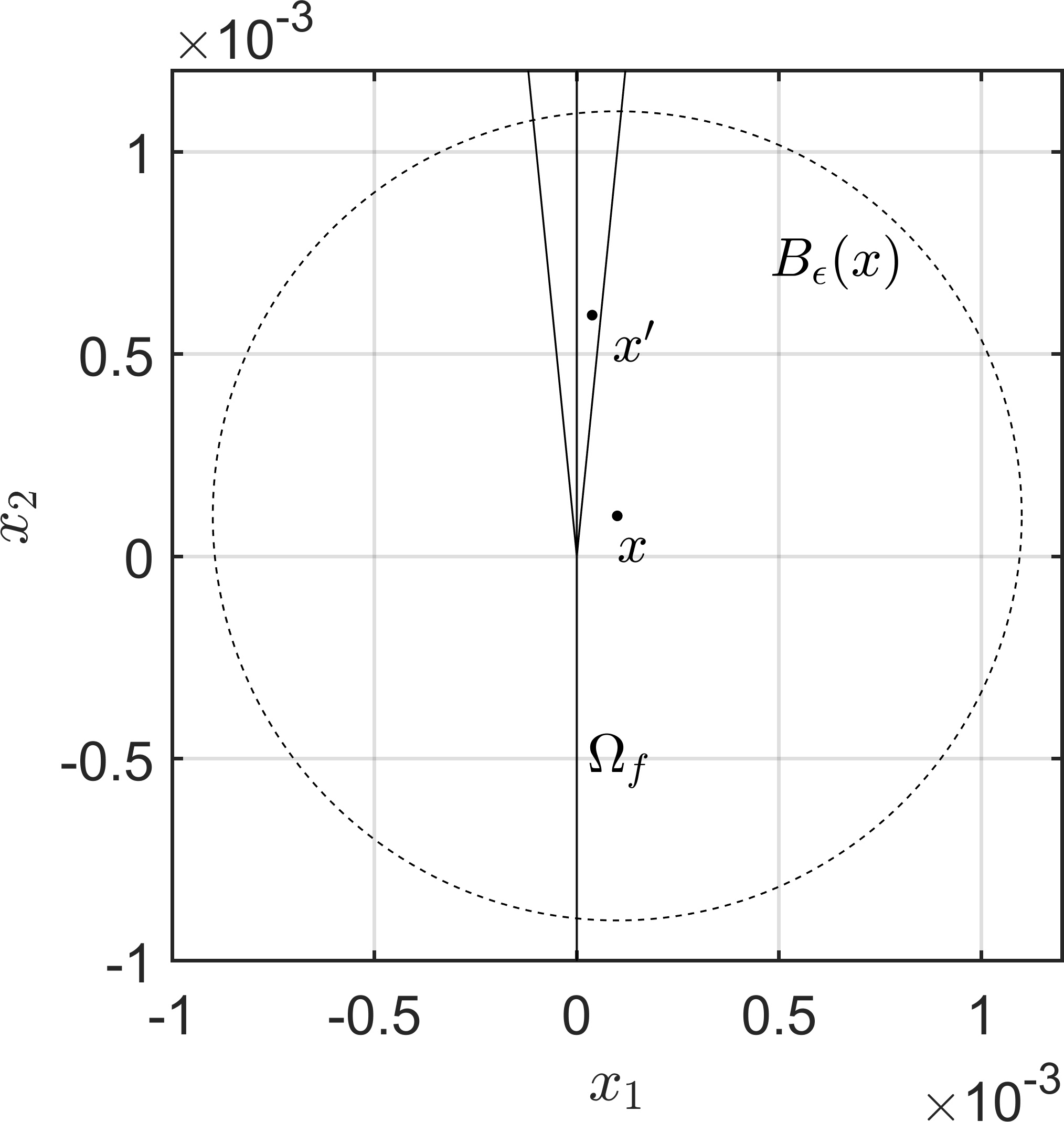}\\
				\textbf{(a)}
			}
			\parbox[b]{0.49\textwidth}{
				\centering 
				\includegraphics[width=0.45\textwidth]{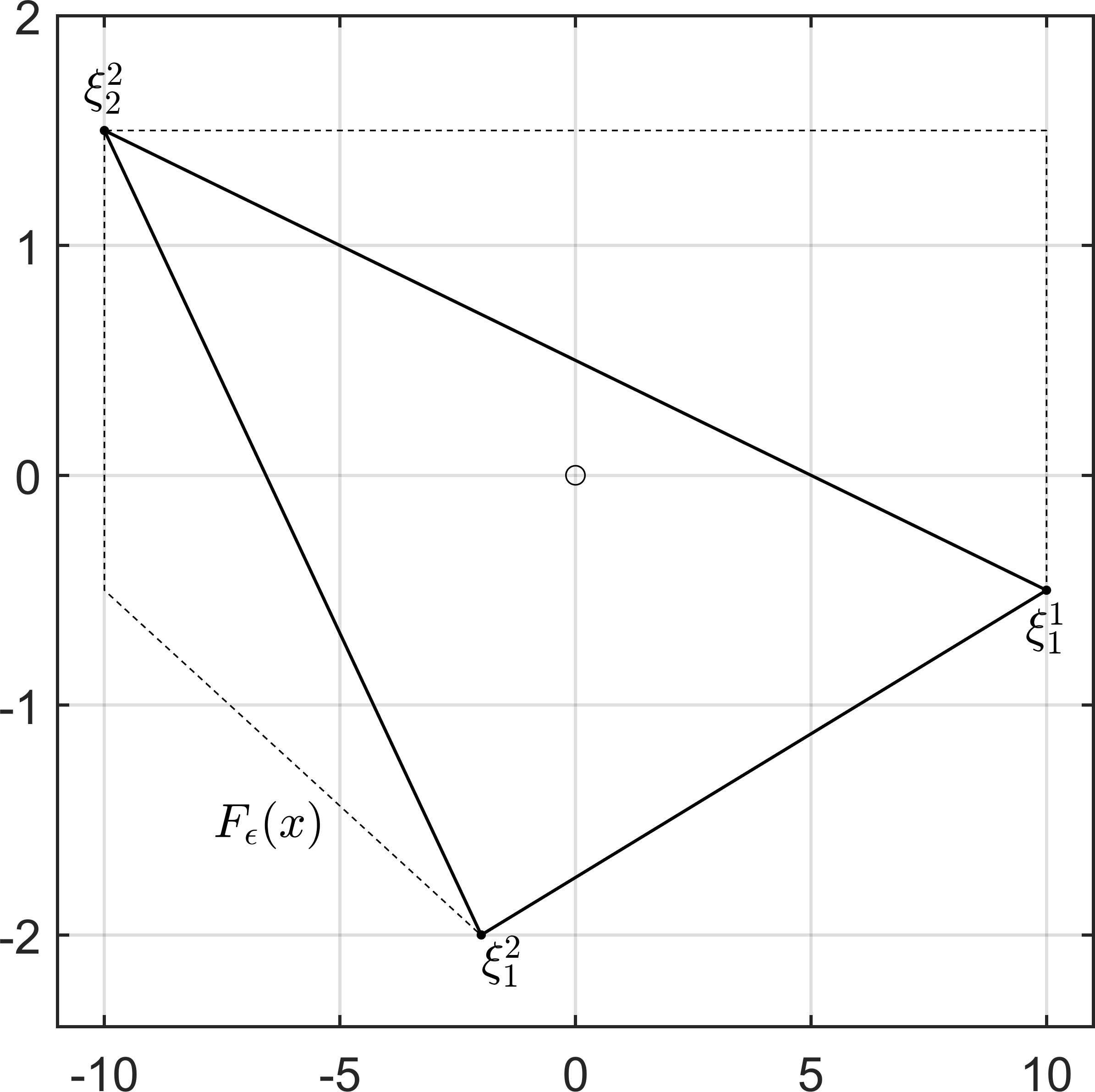}\\
				\textbf{(b)} 
			}
			\caption{\textbf{(a)} The set of nondifferentiable points $\Omega_f$ of $f$, the ball $B_\varepsilon(x)$ for the $\varepsilon$-subdifferential and the point in which subgradients where computed for Algorithm \ref{algo:descent_direction} in Example \ref{exam:compl_subdiff}. \textbf{(b)} The approximation of $F_\varepsilon(x)$ in Algorithm \ref{algo:descent_direction}.}
			\label{fig:example_compl_subdiff}
		\end{figure}
	\end{example}	
	
	Building on Algorithm \ref{algo:descent_direction}, it is now straightforward to construct the descent method for locally Lipschitz continuous MOPs given in Algorithm \ref{algo:nonsmooth_descent_method}. In step 4, the classical Armijo backtracking line search was used (cf.~\cite{FS2000}) for the sake of simplicity. Note that it is well defined due to step 4 in Algorithm \ref{algo:descent_direction}.
	
	\begin{algorithm} 
		\caption{Nonsmooth descent method}
		\label{algo:nonsmooth_descent_method}
		\begin{algorithmic}[1] 
			\REQUIRE Initial point $x_1 \in \R^n$, tolerances $\varepsilon, \delta > 0$, Armijo parameters $c \in (0,1), t_0 > 0$.
			\STATE Set $j = 1$.
			\STATE Compute a descent direction $v_j$ via Algorithm \ref{algo:descent_direction}.
			\STATE If $\| v_j \| \leq \delta$ then stop.
			\STATE Compute
				\begin{align*}
					\bar{s} = \inf(\{ s \in \N \cup \{ 0 \} : f_i(x_j + 2^{-s} t_0 v_j) \leq f_i(x_j) - 2^{-s} t_0 c \| v_j \|^2 \ \forall i \in \{1,...,k\} \})
				\end{align*}
				and set $\bar{t} = \max( \{ 2^{-\bar{s}} t_0, \frac{\varepsilon}{\| v_j \|} \} )$.
			\STATE Set $x_{j+1} = x_j + \bar{t} v_j$, $j = j+1$ and go to step 2.
		\end{algorithmic} 
	\end{algorithm}	
	
	Since we introduced the two tolerances $\varepsilon$ (for the $\varepsilon$-subdifferential) and $\delta$ (as a threshold for when we consider $\varepsilon$-subgradients to be zero), we can not expect that Algorithm \ref{algo:nonsmooth_descent_method} computes points which satisfy the optimality condition \eqref{eq:KKT}. This is why we introduce the following definition, similar to the definition of $\varepsilon$-stationarity from \cite{BLO2005}.
	
	\begin{definition}
		Let $x \in \R^n$, $\varepsilon > 0$ and $\delta > 0$. Then $x$ is called $(\varepsilon, \delta)$\emph{-critical}, if
		\begin{align*}
			\min_{v \in -F_\varepsilon(\bar{x})} \| v \| \leq \delta.
		\end{align*}
	\end{definition}	
	
	It is easy to see that $(\varepsilon, \delta)$-criticality is a necessary optimality condition for Pareto optimality, but a weaker one than \eqref{eq:KKT}. The following theorem shows that convergence in the sense of $(\varepsilon, \delta)$-criticality is what we can expect from our descent method.
	
	\begin{theorem}
		Let $(x_j)_j$ be the sequence generated by Algorithm \ref{algo:nonsmooth_descent_method}. Then either $(f_i(x_j))_j$ is unbounded below for each $i \in \{1,...,k\}$, or $(x_j)_j$ is finite with the last element being $(\varepsilon,\delta)$-critical.			 
	\end{theorem}			
	\begin{proof}
		Assume that $(x_j)_j$ is infinite. Then $\| v_j \| > \delta$ for all $j \in \N$. By step 4 and Lemma \ref{lem:step_size_bound} we have 
		\begin{align*}
			f_i(x_j + \bar{t} v_j) -f_i(x_j) \leq - \bar{t} \| v_j \|^2 \leq - \varepsilon \| v_j \| < -\varepsilon \delta < 0
		\end{align*}
		for all $i \in \{1,...,k\}$. This implies that $(f_i(x_j))_j$ is unbounded below for each $i \in \{1,...,k\}$. \\
		Now assume that $(x_j)_j$ is finite, with $\bar{x}$ and $\bar{v}$ being the last elements of $(x_j)_j$ and $(v_j)_j$, respectively. Since the algorithm stopped, we must have $\| \bar{v} \| \leq \delta$. From the application of Algorithm \ref{algo:descent_direction} in step 2, we know that there must be some $\overline{W} \subseteq F_\varepsilon(\bar{x})$ such that $\bar{v} = \argmin_{v \in -\overline{W}} \| v \|^2$. This implies
		\begin{align*}
			\min_{v \in -F_\varepsilon(\bar{x})} \| v \| \leq \min_{v \in -\conv(\overline{W})} \| v \| = \| \bar{v} \| \leq \delta. 
		\end{align*}		 
	\end{proof}		
		
	\section{Numerical examples} \label{sec:numerical_examples}
	
	In this section we will apply our nonsmooth descent method (Algorithm \ref{algo:nonsmooth_descent_method}) to several examples. We will begin by discussing its typical behavior before comparing its performance to the \emph{multiobjective proximal bundle method} \cite{MKW2014}. Finally, we will combine our method with the \emph{subdivision algorithm} \cite{DSH2005} in order to approximate the entire Pareto set of nonsmooth MOPs.
	
	\subsection{Typical behavior}
	In smooth areas, the behavior of Algorithm \ref{algo:nonsmooth_descent_method} is almost identical to the behavior of the multiobjective steepest descent method \cite{FS2000}. The only difference stems from the fact that, unlike the Clarke subdifferential, the $\varepsilon$-subdifferential does not reduce to the gradient when $f$ is continuously differentiable. As a result, Algorithm \ref{algo:nonsmooth_descent_method} may behave differently in points $x \in \R^n$ where 
	\begin{align*}
		\max \{ \| \nabla f_i(x) - \nabla f_i(y) \| : y \in B_\varepsilon(x), \ i \in \{1,...,k\} \}
	\end{align*}
	is large. (If $f$ is twice differentiable, this can obviously be characterized in terms of second order derivatives.) Nevertheless, if $\varepsilon$ is chosen small enough, this difference can be neglected. Thus, in the following, we will focus on the behavior with respect to the nonsmoothness of $f$.
	
	To show the typical behavior of Algorithm \ref{algo:nonsmooth_descent_method}, we consider the objective function
	\begin{align} \label{eq:MOP_14_15}
		f : \R^2 \rightarrow \R^2, \quad x \mapsto 
		\begin{pmatrix}
			\max \{ x_1^2 + (x_2 - 1)^2 + x_2 - 1, -x_1^2 - (x_2 - 1)^2 + x_2 + 1 \} \\
			-x_1 + 2 (x_1^2 + x_2^2 - 1) + 1.75 | x_1^2 + x_2^2 - 1 |
		\end{pmatrix}
	\end{align}
	from \cite{MKW2014} (combining \emph{Crescent} from \cite{K1985a} and \emph{Mifflin 2} from \cite{MN1992}). The set of nondifferentiable points is $\Omega_f = S^1 \cup (S^1 + (0,1)^\top)$. We consider the starting points 
	\begin{align*}
		x^1 = (0,-0.3)^\top, \quad x^2 = (0.6,1.0)^\top, \quad x^3 = (-1,-0.2)^\top,
	\end{align*}
	the tolerances $\varepsilon = 10^{-3}$, $\delta = 10^{-3}$ and the Armijo parameters $c = 0.25$, $t_0 = 1$. The results are shown in Figure \ref{fig:example_typical_behavior}.
	\begin{figure}[ht]
		\centering
		\includegraphics[width=0.7\textwidth]{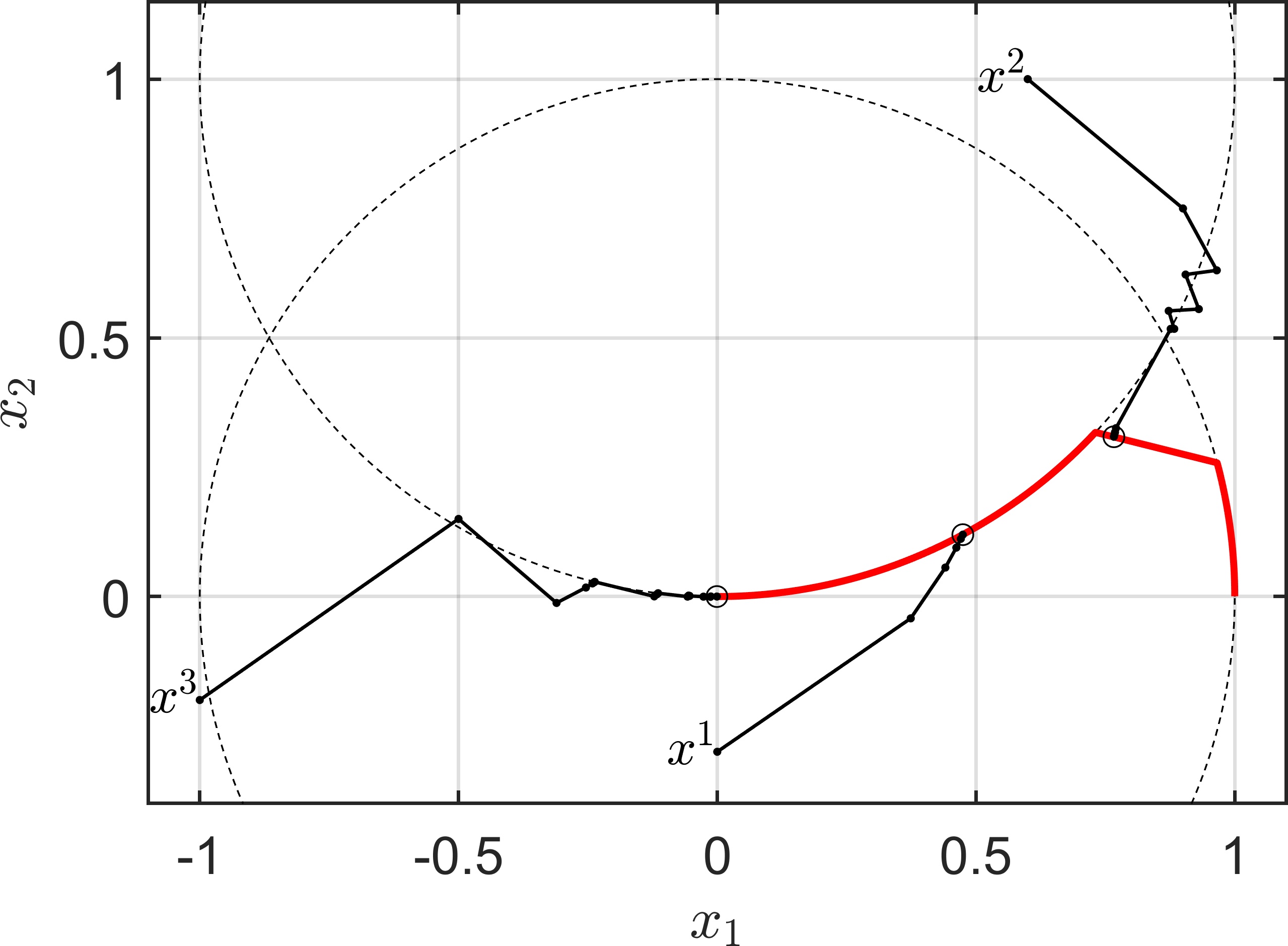}
		\caption{Result of Algorithm \ref{algo:nonsmooth_descent_method} in three different starting points for the MOP \eqref{eq:MOP_14_15}. The Pareto set is shown in red, the dashed lines show the set of nondifferentiable points $\Omega_f$.}
		\label{fig:example_typical_behavior}
	\end{figure}
	We will briefly go over the result for each starting point:
	\begin{itemize}
		\item For $x^1$, the sequence moves through the smooth area like the steepest descent method until a point is found with a distance less or equal $\varepsilon$ to the set of nondifferentiable points $\Omega_f$. In that point, more than one $\varepsilon$-subgradient is required to obtain a sufficient approximation of the $\varepsilon$-subdifferentials. Since this part of $\Omega_f$ is Pareto optimal, no acceptable descent direction (cf.~\eqref{eq:accept_direction}) is found and the algorithm stops (in a $(\varepsilon,\delta)$-critical point).
		\item For $x^2$, the sequence starts zig-zagging around the non-optimal part of $\Omega_f$, since the points are too far away from $\Omega_f$ for the algorithm to notice the nondifferentiability. When a point is found with distance less or equal $\varepsilon$ to $\Omega_f$, a better descent direction is found, breaking the zig-zagging motion.
		\item For $x^3$, the sequence has a similar zig-zagging motion to the previous case. The difference is that this time, the sequence moves along $\Omega_f$ until a Pareto optimal point in $\Omega_f$ is found.  
	\end{itemize}	
	
	As described above, the zig-zagging behavior when starting in $x^2$ is caused by the fact that $\varepsilon$ was too small for the method to notice the nondifferentiability. To circumvent problems like this and quickly move through problematic areas, it is possible to apply Algorithm \ref{algo:nonsmooth_descent_method} consecutively with decreasing values of $\varepsilon$. The result is Algorithm \ref{algo:eps_decr_descent_method}. (A similar idea was implemented in \cite{MY2012}.)
	
	\begin{algorithm} 
		\caption{$\varepsilon$-decreasing nonsmooth descent method}
		\label{algo:eps_decr_descent_method}
		\begin{algorithmic}[1] 
			\REQUIRE Initial point $x_1 \in \R^n$, tolerances $\delta, \varepsilon_1, ..., \varepsilon_K > 0$ , Armijo parameters $c \in (0,1), t_0 > 0$.
			\STATE Set $y_1 = x_1$.
			\FOR{$i = 1, ..., K$}
			\STATE Apply Algorithm \ref{algo:nonsmooth_descent_method} with initial point $y_i$ and tolerance $\varepsilon = \varepsilon_i$. Let $y_{i+1}$ be the final element in the generated sequence.
			\ENDFOR
		\end{algorithmic} 
	\end{algorithm}		
		
	\subsection{Comparison to the multiobjective proximal bundle method}

	We will now compare Algorithms \ref{algo:nonsmooth_descent_method} and \ref{algo:eps_decr_descent_method} to the \emph{multiobjective proximal bundle method} (MPB) by M\"akel\"a, Karmitsa and Wilppu from \cite{MKW2014} (see also \cite{M2003}). As test problems, we consider the 18 MOPs in Table \ref{table:test_problems}, which are created on the basis of the scalar problems from \cite{MKW2014}. Problems 1 to 15 are convex (and were also considered in \cite{MKM2018}) and problems 16 to 18 are nonconvex. Due to their simplicity, we are able to differentiate all test problems by hand to obtain explicit formulas for the subgradients. For each test problem, we choose 100 starting points on a 10 $\times$ 10 grid in the corresponding area given in Table \ref{table:test_problems}.
	
	\begin{table}
		\centering
		\caption{Test problems (using objectives from \cite{MKW2014})}
		\label{table:test_problems}
		\begin{small}
			\begin{tabular}{| c | l | c || c | l | c |}
				\hline
				Nr. & $f_i$ & Area & Nr. & $f_i$ & Area \\
				\hline
				1. & CB3, DEM & $[-3,3]^2$ & 10. & QL, LQ & $[-3,3]^2$  \\
				2. & CB3, QL & $[-3,3]^2$ & 11. & QL, Mifflin 1 & $[-3,3]^2$ \\
				3. & CB3, LQ & $[0.5,1.5]^2$ & 12. & QL, Wolfe & $[-3,3]^2$ \\
				4. & CB3, Mifflin 1 & $[-3,3]^2$ & 13. & LQ, Mifflin 1 & $[0.5,1.5] \times [-0.5,1]$ \\
				5. & CB3, Wolfe & $[-3,3]^2$ & 14. & LQ, Wolfe & $[-3,3]^2$ \\
				6. & DEM, QL & $[-3,3]^2$ & 15. & Mifflin 1, Wolfe & $[-3,3]^2$ \\ \cline{4-6}
				7. & DEM, LQ & $[-3,3]^2$ & 16. & Crescent, Mifflin 2 & $[-0.5,1.5]^2$ \\
				8. & DEM, Mifflin 1 & $[-3,3]^2$ & 17. & Mifflin 2, WF & $[-3,3]^2$ \\
				9. & DEM, Wolfe & $[-3,3]^2$ & 18. & Mifflin 2, SPIRAL & $[-3,3]^2$ \\
				\hline
			\end{tabular}
		\end{small}			
	\end{table}	
	
	For the MPB, we use the Fortran implementation from \cite{M2003} with the default parameters. For Algorithm \ref{algo:nonsmooth_descent_method}, we use $\varepsilon = 10^{-3}$, $\delta = 10^{-3}$, $c = 0.25$ and $t_0 = \max\{ \| v_j \|^{-1}, 1 \}$ (i.e., the initial step size $t_0$ is chosen depending on the norm of the descent direction $v_j$ in the current iteration). For Algorithm \ref{algo:eps_decr_descent_method}, we additionally use $\varepsilon_1 = 10^{-1}$, $\varepsilon_2 = 10^{-2}$, $\varepsilon_3 = 10^{-3}$. By this choice of parameters, all three methods produce results of similar approximation quality.
	
	To compare the performance of the three methods, we count the number of evaluations of objectives $f_i$, their subgradients $\xi \in \partial f_i$ and the number of iterations (i.e., descent steps) needed. (This means that one call of $f$ will account for $k$ evaluations of objectives.) Since the MPB always evaluates all objectives and subgradients in a point, the value for the objectives and the subgradients are the same here. The results are shown in Table \ref{table:performance} and are discussed in the following.
	\begin{table}
		\centering
		\caption{Performance of MPB, Algorithm \ref{algo:nonsmooth_descent_method} and Algorithm \ref{algo:eps_decr_descent_method} for the test problems in Table \ref{table:test_problems} for $100$ starting points}
		\label{table:performance}
		\begin{small}
			\begin{tabular}{|c !{\vrule width 2pt} c | c | c !{\vrule width 2pt} c | c | c !{\vrule width 2pt} c | c | c |}
				\hline
				& \multicolumn{3}{c !{\vrule width 2pt}}{\#$f_i$} & \multicolumn{3}{c !{\vrule width 2pt}}{\#$\partial f_i$} & \multicolumn{3}{c |}{\# Iter} \\
				\hline
				Nr. & MPB & Alg. \ref{algo:nonsmooth_descent_method} & Alg. \ref{algo:eps_decr_descent_method} & MPB & Alg. \ref{algo:nonsmooth_descent_method} & Alg. \ref{algo:eps_decr_descent_method} & MPB & Alg. \ref{algo:nonsmooth_descent_method} & Alg. \ref{algo:eps_decr_descent_method} \\
				\hline
				1. & \textbf{1780} & 6924 & 7801 & 1780 & \textbf{1102} & 1751 & 761 & \textbf{492} & 695 \\
				2. & \textbf{2522} & 14688 & 12263 & 2522 & \textbf{1906} & 2351 & 1151 & \textbf{842} & 914 \\
				3. & \textbf{880} & 5625 & 6447 & \textbf{880} & 921 & 1534 & \textbf{340} & 448 & 662 \\
				4. & \textbf{4416} & 103826 & 17664 & 4416 & 11774 & \textbf{3415} & 1832 & 4644 & \textbf{1242} \\
				5. & \textbf{2956} & 30457 & 16877 & \textbf{2956} & 3479 & 3037 & 1377 & 1616 & \textbf{1161} \\
				6. & \textbf{1640} & 8357 & 8684 & 1640 & \textbf{1209} & 1802 & 706 & \textbf{552} & 736 \\
				7. & \textbf{1702} & 8736 & 8483 & 1702 & \textbf{1307} & 1832 & 723 & \textbf{595} & 739 \\
				8. & \textbf{4226} & 8283 & 8620 & 4226 & \textbf{1318} & 1914 & 1204 & \textbf{582} & 759 \\
				9. & \textbf{1828} & 8201 & 8794 & 1828 & \textbf{1194} & 1805 & 793 & \textbf{536} & 732 \\
				10. & \textbf{1782} & 6799 & 7201 & 1782 & \textbf{1101} & 1722 & 684 & \textbf{543} & 733 \\
				11. & \textbf{4426} & 52096 & 17594 & 4426 & 6311 & \textbf{3189} & 1964 & 2442 & \textbf{1206} \\
				12. & \textbf{2482} & 15146 & 12446 & 2482 & \textbf{1992} & 2401 & 1140 & \textbf{967} & 1010 \\
				13. & \textbf{2662} & 36570 & 9513 & 2662 & 4958 & \textbf{2247} & 1221 & 1692 & \textbf{787} \\
				14. & \textbf{4264} & 95303 & 12227 & 4264 & 9524 & \textbf{2571} & 1774 & 4379 & \textbf{921} \\
				15. & \textbf{3594} & 85936 & 15669 & 3594 & 9329 & \textbf{3124} & 1444 & 3963 & \textbf{1125} \\
				\hline
				16. & \textbf{2206} & 20372 & 11094 & \textbf{2206} & 2596 & 2400 & \textbf{884} & 1194 & 947 \\
				17. & \textbf{2388} & 7920 & 5852 & 2388 & \textbf{1272} & 1556 & 868 & \textbf{626} & 706 \\
				18. & \textbf{11430} & 166707 & 31528 & 11430 & 16676 & \textbf{6902} & 2789 & 8291 & \textbf{2412} \\
				\hline
				Avg. & \textbf{3176.9} & 37885.9 & 12153.2 & 3176.9 & 4331.6 & \textbf{2530.7} & 1203.1 & 1911.3 & \textbf{971.5} \\
				     & 100\% & 1192.5\% & 382.5\% & 100\% & 136.3\% & 79.7\% & 100\% & 158.9\% & 80.8\% \\
				\hline
			\end{tabular}	
		\end{small}
	\end{table}
	\begin{itemize}
		\item \textbf{Function evaluations:} When considering the number of function evaluations, it is clear that the MPB requires far less evaluations than both of our algorithms. In our methods, these evaluations are used to check if a descent direction is acceptable (cf.~\eqref{eq:accept_direction}) and for the computation of the Armijo step length. One reason for the larger total amount is the fact that unlike the MPB, our methods are autonomous in the sense that they do not reuse information from previous iterations, so some information is potentially gathered multiple times. Additionally, the step length we use is fairly simple, so it might be possible to lower the number of evaluations by using a more sophisticated step length. When comparing our methods to each other, we see that Algorithm \ref{algo:eps_decr_descent_method} is a lot more efficient than Algorithm \ref{algo:nonsmooth_descent_method} when the number of evaluations is high and is slightly less efficient when the number of evaluations is low. The reason for this is that for simple problems (i.e., where the number of evaluations is low), some of the iterations of Algorithm \ref{algo:eps_decr_descent_method} will be redundant, because the $(\varepsilon_{i-1},\delta)$-critical point of the previous iteration is already $(\varepsilon_i,\delta)$-critical.
		\item \textbf{Subgradient evaluations:} For the subgradient evaluations, we see that MPB is slightly superior to our methods on problems 3, 5 and 16, but inferior on the rest. Regarding the comparison of Algorithms \ref{algo:nonsmooth_descent_method} and \ref{algo:eps_decr_descent_method}, we observe the same pattern as for the function evaluations: Algorithm \ref{algo:nonsmooth_descent_method} is superior for simple and Algorithm \ref{algo:eps_decr_descent_method} for complex problems.
		\item \textbf{Iterations:} For the number of iterations, besides problem 5, we see the exact same pattern as for the number of subgradient evaluations. Note that the MPB can perform \emph{null steps}, which are iterations where only the bundle is enriched, while the current point in the descent sequence stays the same.
	\end{itemize}		
	
	For our set of test problems, this leads us to the overall conclusion that in terms of function evaluations, the MPB seems to be superior to our methods, while in terms of subgradient evaluations, our methods seem to be (almost always) more efficient. Furthermore, we remark that the implementation of the MPB is somewhat challenging, whereas our method can be implemented relatively quickly.
	
	\subsection{Combination with the subdivision algorithm}
	Note that so far, we have a method where we can put in some initial point from $\R^n$ and obtain a single $(\varepsilon,\delta)$-critical point (close to an actual Pareto optimal point) as a result. But ultimately, we are not interested in one, but all Pareto optimal points. The intuitive and straightforward approach to extend our method would be to just take a large set of well-spread initial points and apply our method to each of them. The problem with this is that we have no guarantee that this results in a good approximation of the Pareto set. To solve this issue, we combine our method with the \emph{subdivision algorithm} which was developed for smooth problems in \cite{DSH2005}. Since we only have to do minor adjustments for the nonsmooth case, we will only sketch the method here and refer to \cite{DSH2005} for the details. 
	
	The idea is to interpret the nonsmooth descent method as a discrete dynamical system 
	\begin{align} \label{eq:dyn_system}
		x_{j+1} = g(x_j), \quad j = 0, 1, 2, ..., \quad x_0 \in \R^n,
	\end{align}
	where $g: \R^n \rightarrow \R^n$ is the map that applies one iteration of Algorithm \ref{algo:nonsmooth_descent_method} to a point in $\R^n$. (For the sake of brevity, we have omitted the rest of the input of the algorithm here.) Since no information is carried over between iterations of the algorithm, the trajectory (i.e., the sequence) generated by the system \eqref{eq:dyn_system} is the same as the one generated by Algorithm \ref{algo:nonsmooth_descent_method}. In particular, this means that the Pareto set of the MOP is contained in the set of fixed points of the system \eqref{eq:dyn_system}. Thus, the subdivision algorithm (which was originally designed to compute attractors of dynamical systems) can be used to compute (a superset of) the Pareto set. 
	
	The subdivision algorithm starts with a large hypercube (or \emph{box}) in $\R^n$ that contains the Pareto set and mainly consists of two steps:
	\begin{enumerate}
		\item \textbf{Subdivision:} Divide each box in the current set of boxes into smaller boxes.
		\item \textbf{Selection:} Compute the image of the union of the current set of boxes under $g$ and remove all boxes that have an empty intersection with this image. Go to step 1.
	\end{enumerate}
	In practice, we realize step 1 by evenly dividing each box into $2^n$ smaller boxes and step 2 by using the image of a set of sample points. The algorithm is visualized in Figure~\ref{fig:GAIO}.
	
	\begin{figure}[ht] 
		\centering
		\parbox[b]{0.3\textwidth}{
			\centering 
			\includegraphics[width=0.3\textwidth]{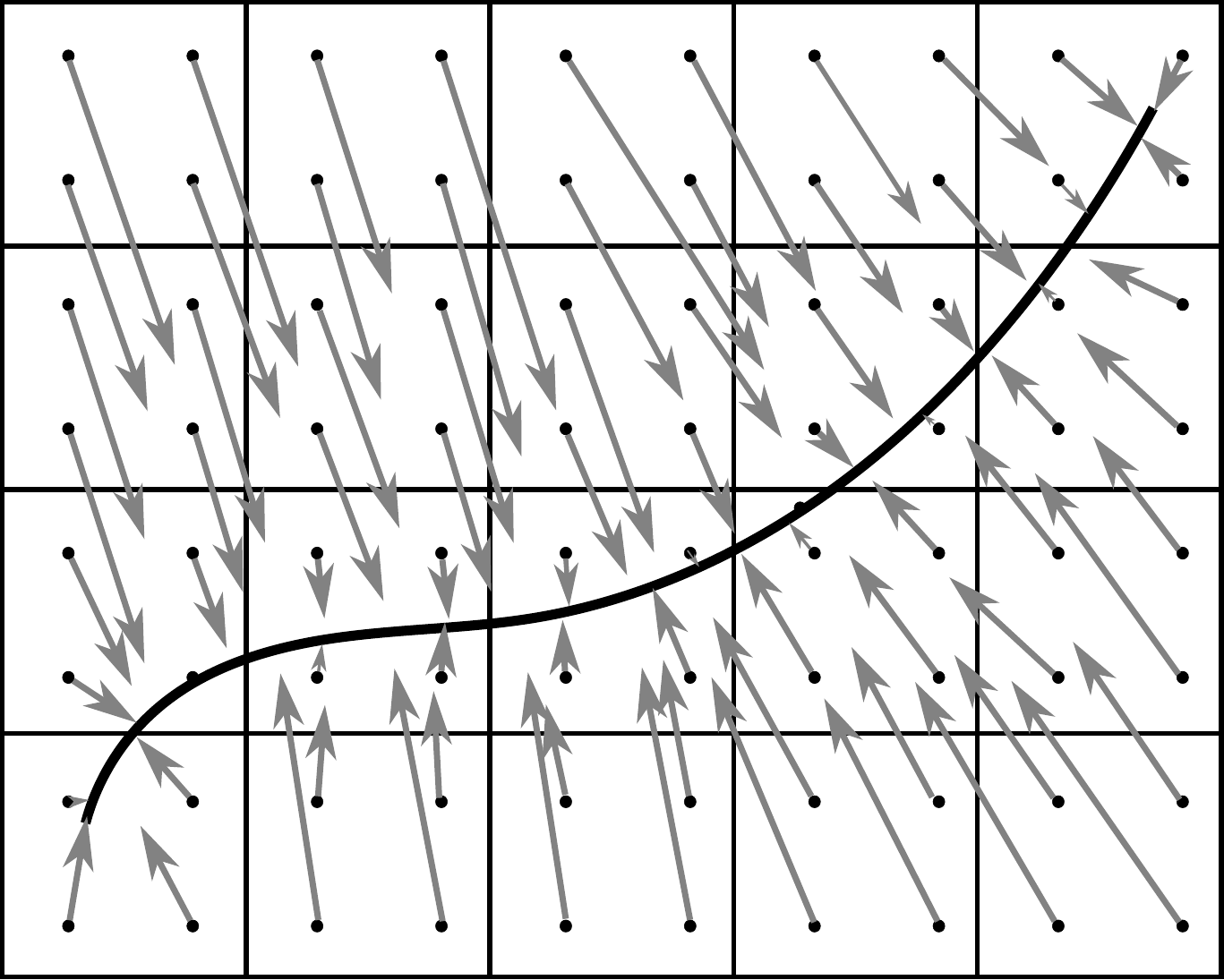}\\
			\textbf{(a)}
		}
		\hfil
		\parbox[b]{0.3\textwidth}{
			\centering 
			\includegraphics[width=0.3\textwidth]{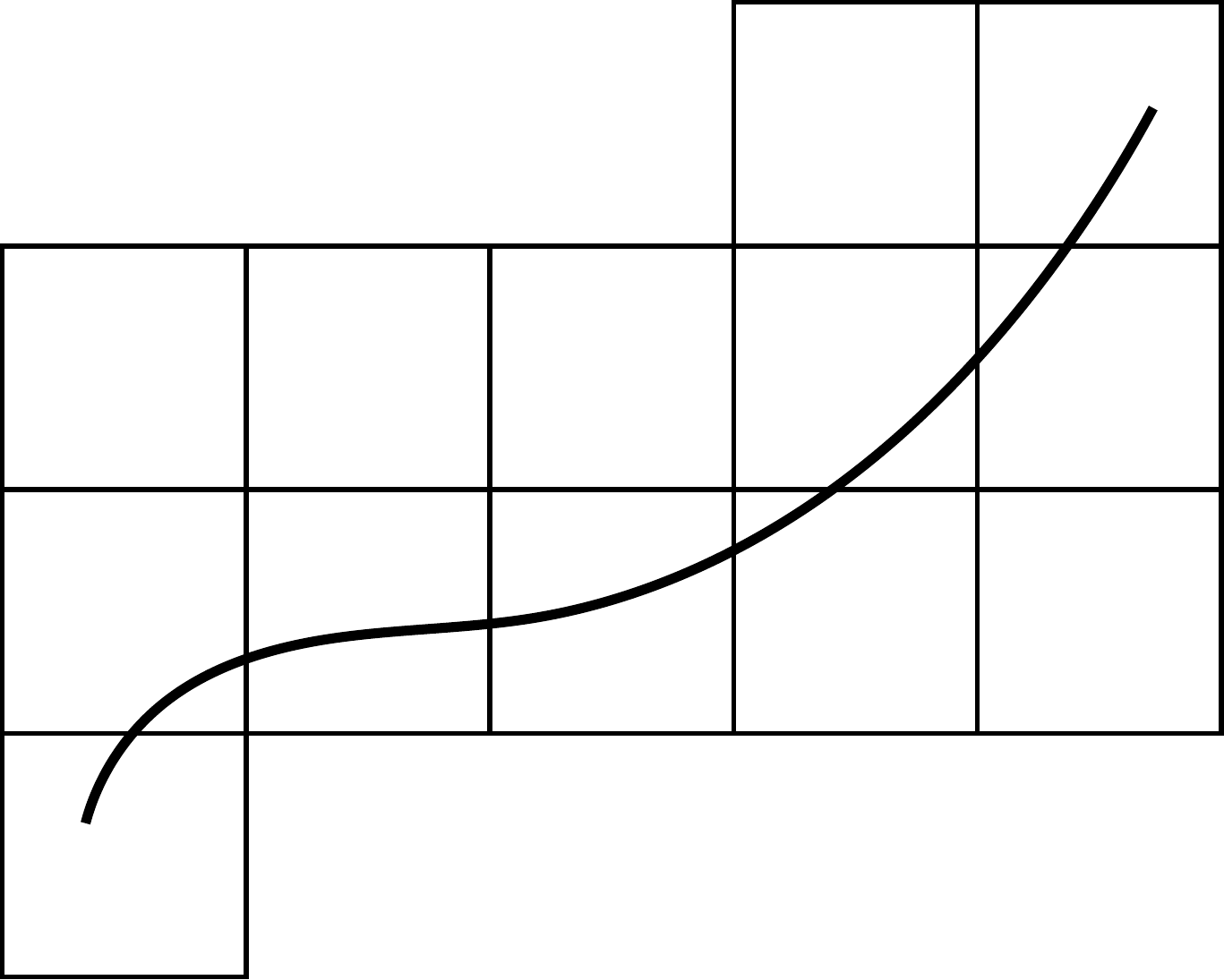}\\
			\textbf{(b)} 
		}
		\caption{Subdivision algorithm. \textbf{(a)} Applying $g$ to a set of sample points. \textbf{(b)} Selection step, where boxes with an empty intersection with the image of $g$ are removed.}
		\label{fig:GAIO}
	\end{figure}
	
	Unfortunately, the convergence results of the subdivison algorithm only apply if $g$ is a diffeomorphism. If the objective function $f$ is smooth, then the descent direction is at least continuous (cf.~\cite{FS2000}) and the resulting dynamical system $g$, while not being a diffeomorphism, still behaves well enough for the subdivision algorithm to work. If $f$ is nonsmooth, then our descent direction is inherently discontinuous close to the nonsmooth points. Thus, the  subdivision algorithm applied to \eqref{eq:dyn_system} will (usually) fail to work. In practice, we were able to solve this issue by applying multiple iterations of Algorithm \ref{algo:nonsmooth_descent_method} in $g$ at once instead of just one. Roughly speaking, this smoothes $g$ by pushing the influence of the discontinuity further away from the Pareto set and was sufficient for convergence (in our tests). 
	
	Figures \ref{fig:subdiv_9_10} to \ref{fig:subdiv_15_16} show the result of the subdivision algorithm for some of the problems from Table \ref{table:test_problems}. For each problem, we used $15$ iterations of Algorithm \ref{algo:nonsmooth_descent_method} in $g$, $[-3.1,3]^2$ as the starting box and applied $9$ iterations of the subdivision algorithm. For the approximation of the Pareto front (i.e., the image of the Pareto set), we evaluated $f$ in all points of the image of $g$ of the last selection step in the subdivision algorithm. In all of these examples, the algorithm produced a tight approximation of the Pareto set. 
	
	\begin{figure}[ht] 
		\parbox[b]{0.49\textwidth}{
			\centering 
			\includegraphics[width=0.325\textwidth]{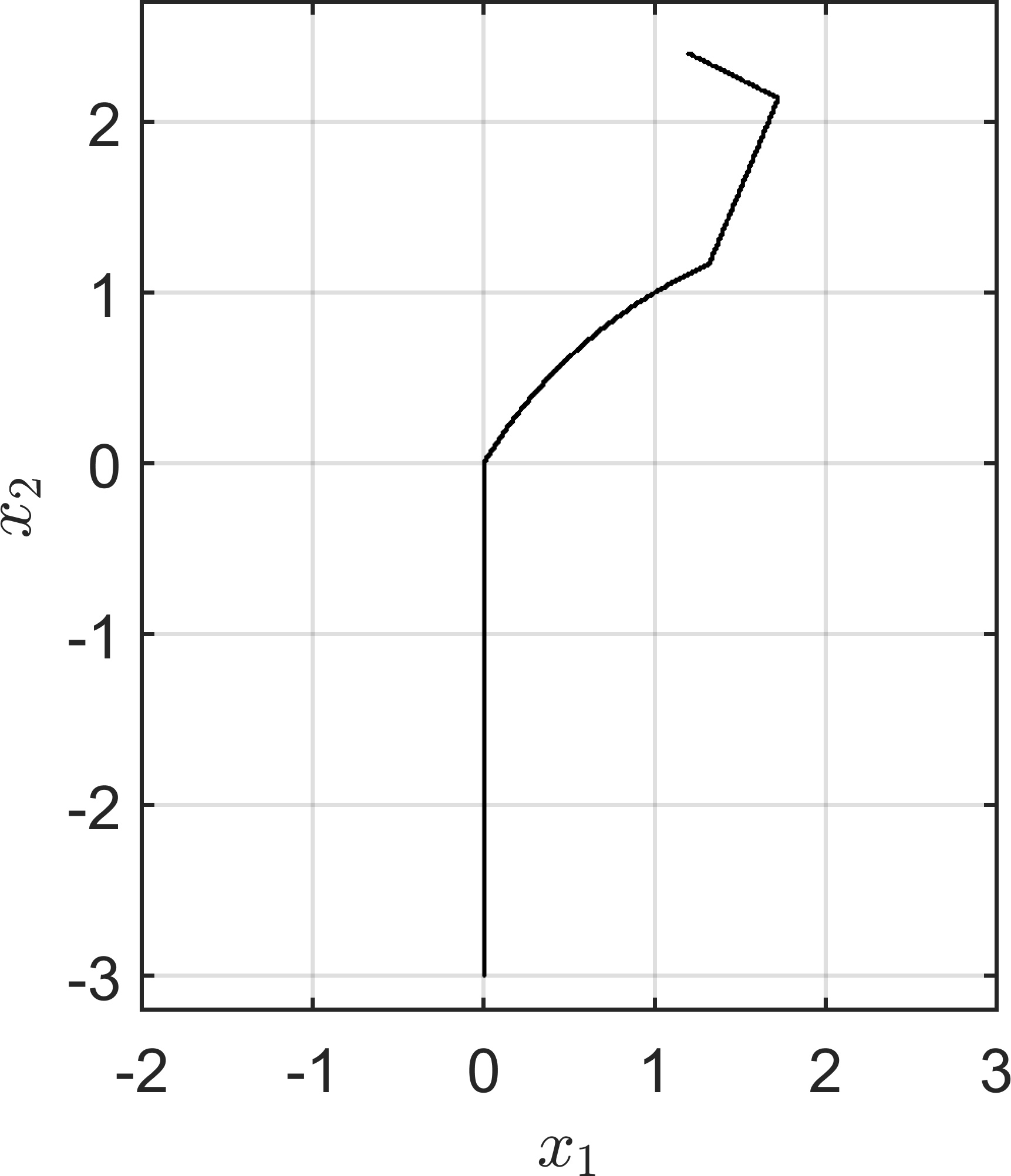}\\
			\textbf{(a)}
		}
		\parbox[b]{0.49\textwidth}{
			\centering 
			\includegraphics[width=0.45\textwidth]{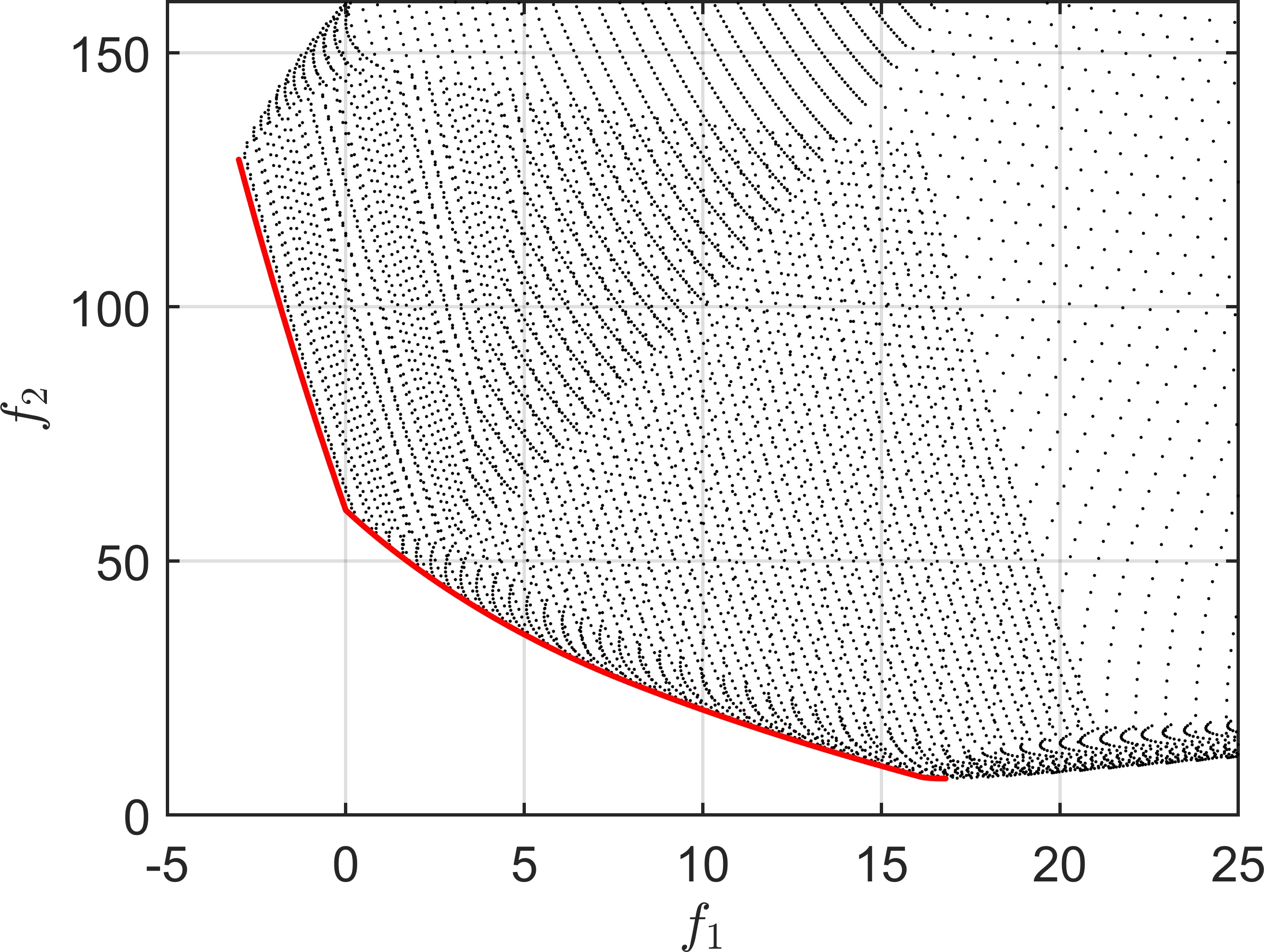}\\
			\textbf{(b)} 
		}
		\caption{\textbf{(a)} Result of the subdivision algorithm applied to problem 6 from Table \ref{table:test_problems}. \textbf{(b)} Corresponding approximation of the Pareto front (red) and a pointwise discretization of the image of $f$ (black).}
		\label{fig:subdiv_9_10}
	\end{figure}	
	
	\begin{figure}[ht] 
		\parbox[b]{0.49\textwidth}{
			\centering 
			\includegraphics[width=0.35\textwidth]{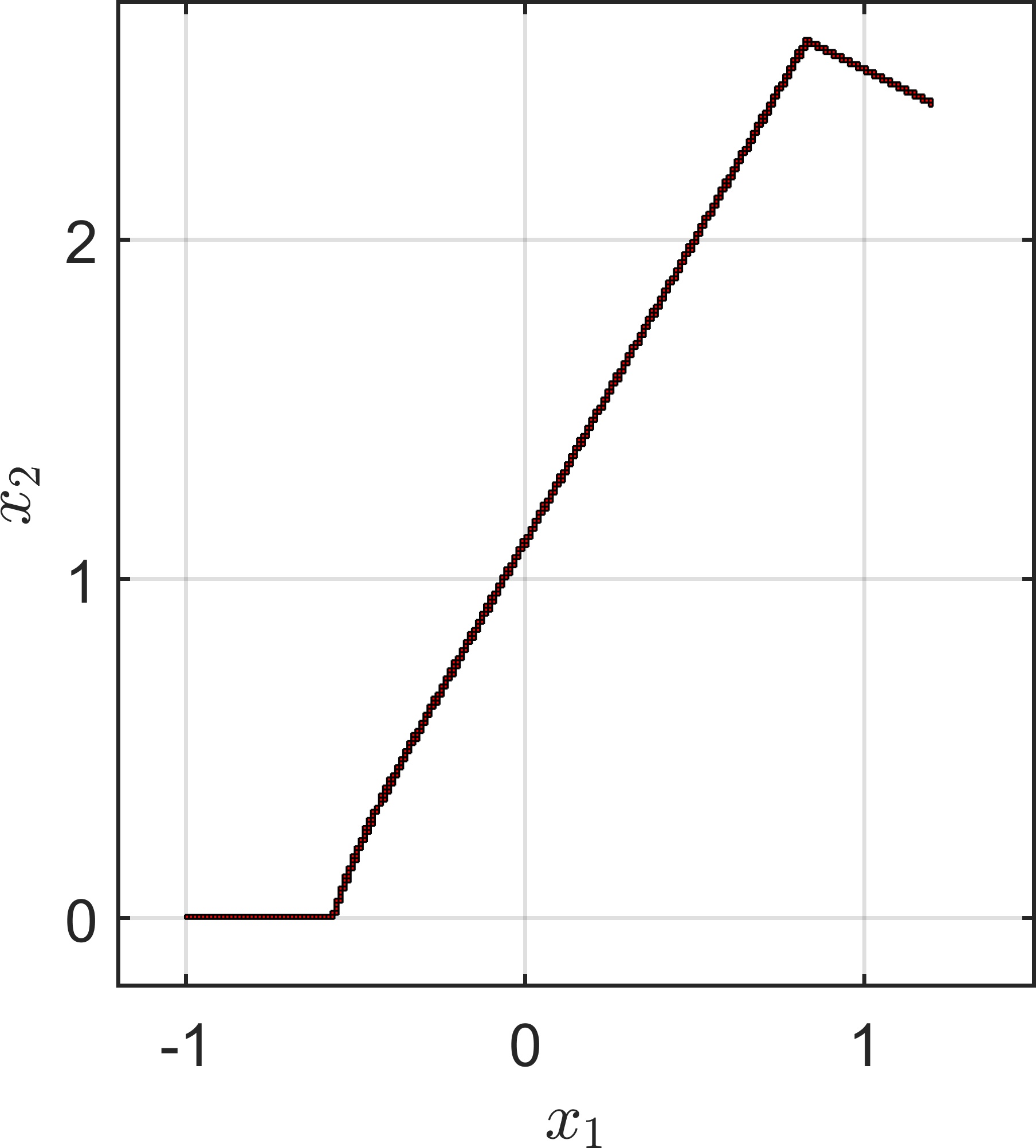}\\
			\textbf{(a)}
		}
		\parbox[b]{0.49\textwidth}{
			\centering 
			\includegraphics[width=0.45\textwidth]{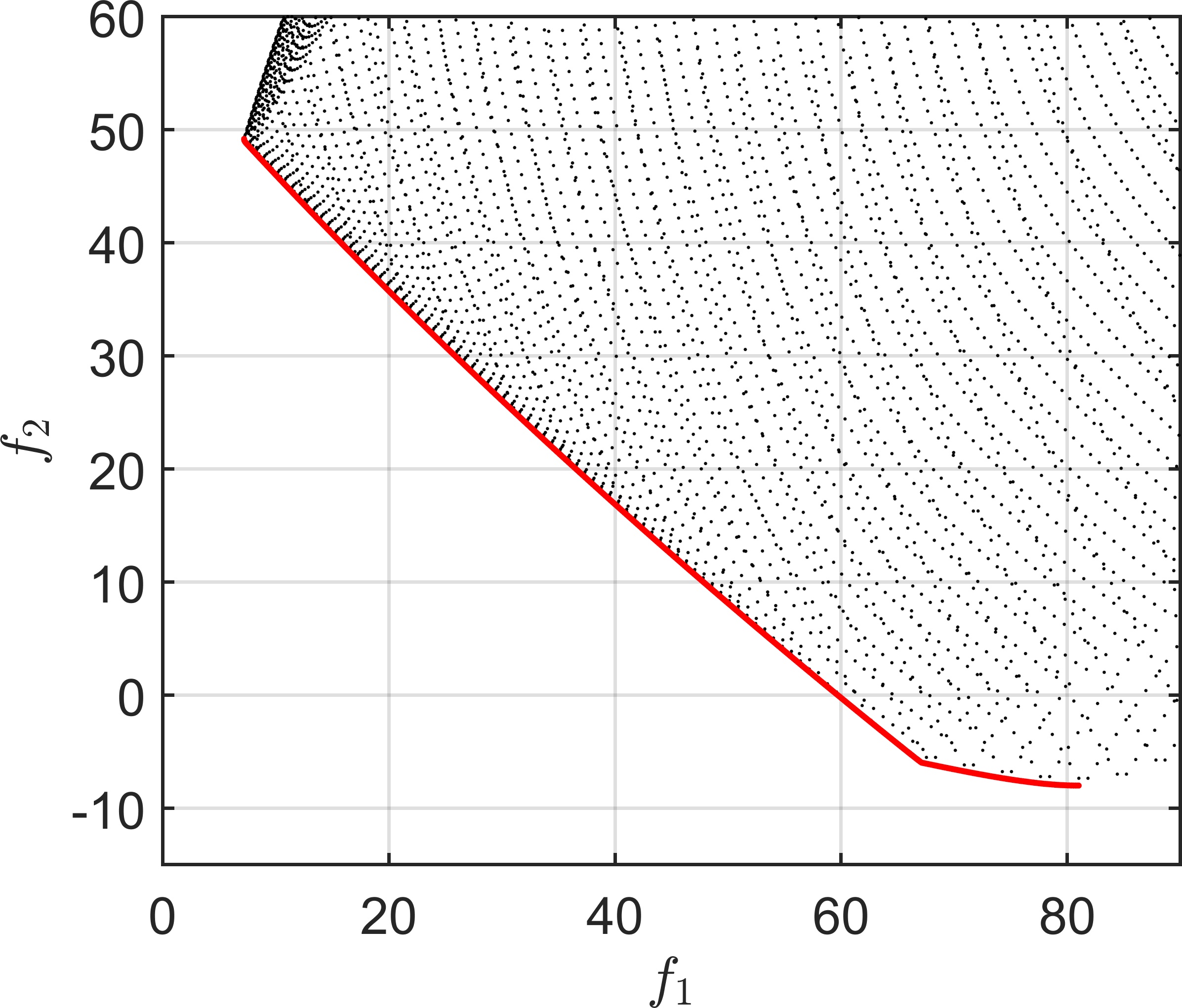}\\
			\textbf{(b)} 
		}
		\caption{\textbf{(a)} Result of the subdivision algorithm applied to problem 12 from Table \ref{table:test_problems}. \textbf{(b)} Corresponding approximation of the Pareto front (red) and a pointwise discretization of the image of $f$ (black).}
		\label{fig:subdiv_10_13}
	\end{figure}
	
	\begin{figure}[ht] 
		\parbox[b]{0.49\textwidth}{
			\centering 
			\includegraphics[width=0.45\textwidth]{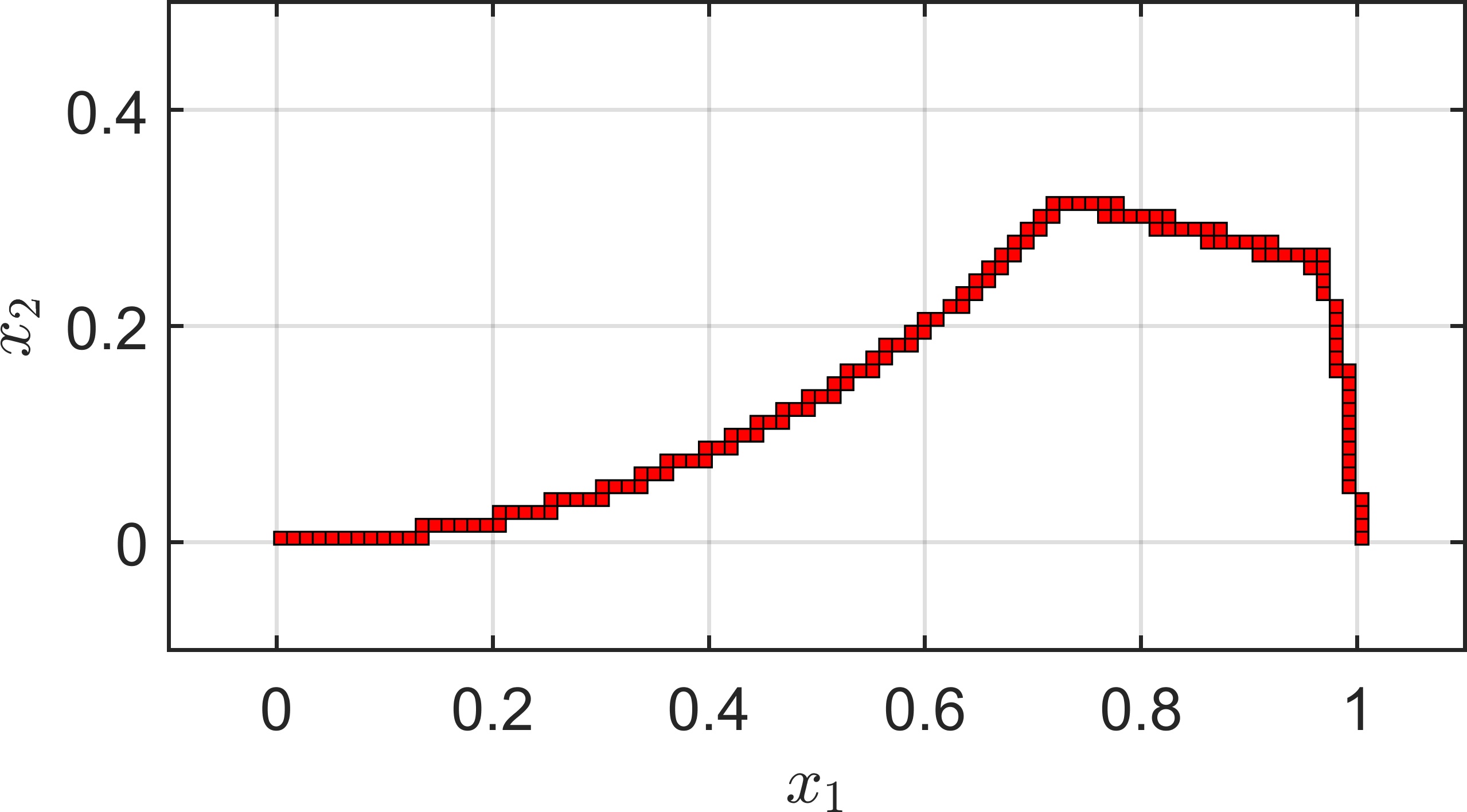}\\
			\textbf{(a)}
		}
		\parbox[b]{0.49\textwidth}{
			\centering 
			\includegraphics[width=0.45\textwidth]{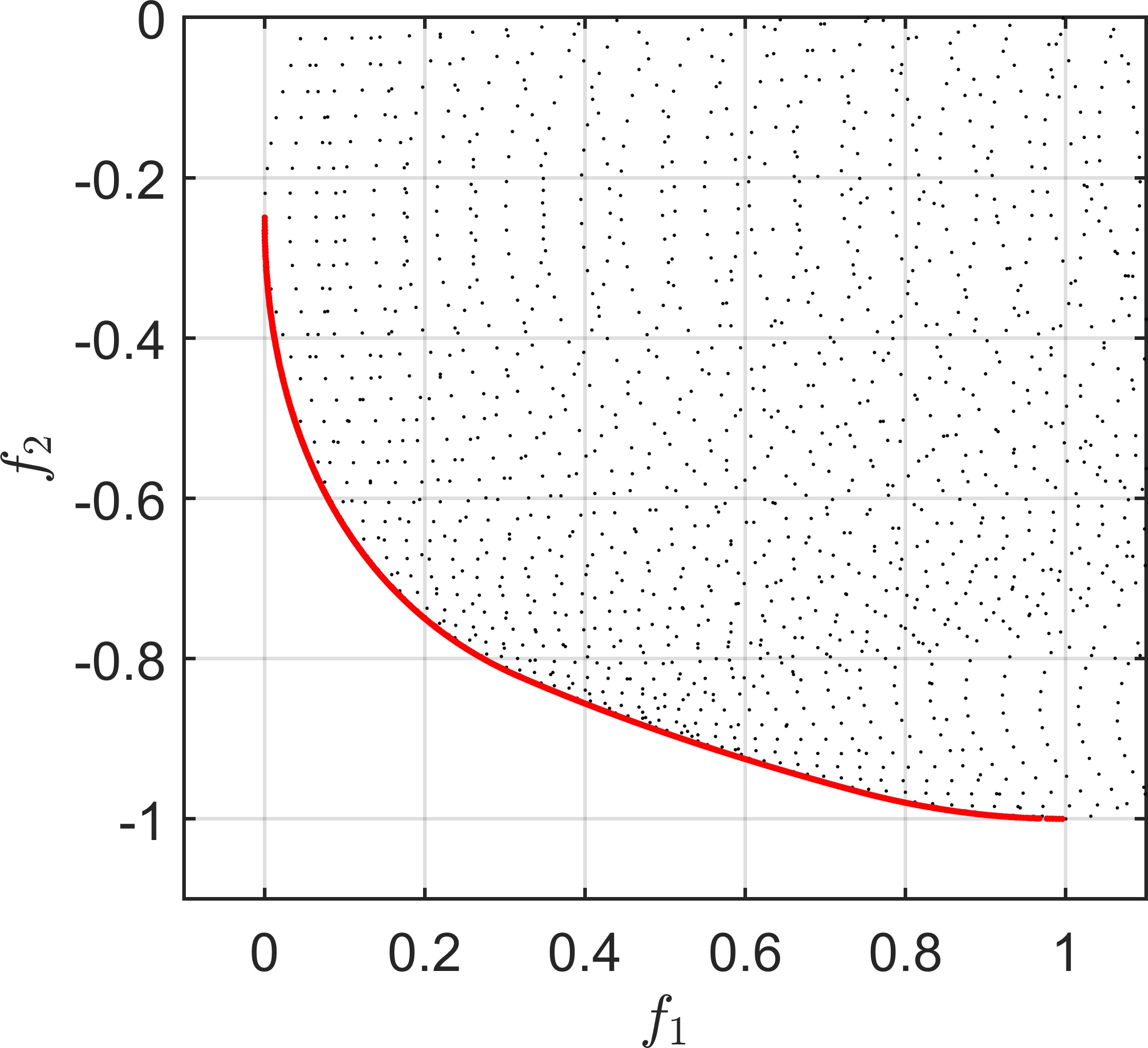}\\
			\textbf{(b)} 
		}
		\caption{\textbf{(a)} Result of the subdivision algorithm applied to problem 16 from Table \ref{table:test_problems}. \textbf{(b)} Corresponding approximation of the Pareto front (red) and a pointwise discretization of the image of $f$ (black).}
		\label{fig:subdiv_15_16}
	\end{figure}	
	
	\section{Conclusion and outlook} \label{sec:conclusion}
	In this article, we have developed a new descent method for locally Lipschitz continuous multiobjective optimization problems, which is based on the efficient approximation of the Clarke subdifferentials of the objective functions from \cite{MY2012}. In \cite{AGG2015}, it was shown that the element with the smallest norm in the negative convex hull of the union of the subdifferentials is a descent direction for all objectives at the same time. In practice, the entire subdifferentials which are required to compute this direction are rarely known and only single subgradients can be computed. To solve this issue, we presented a method to obtain an approximation of the subdifferentials which is sufficient to obtain a descent direction. The idea is to start with a rough approximation of the subdifferentials by only a few subgradients and then systematically enrich the approximation with new subgradients until a direction of sufficient descent is found. By combining the descent direction with an Armijo step length, we obtained a descent method for nonsmooth MOPs and showed convergence to points which satisfy a necessary condition for Pareto optimality. We then compared the performance to the multiobjective proximal bundle method from \cite{MKW2014}. For the 18 test problems we considered, the MPB was superior in terms of objective function evaluations, but our method required less subgradient evaluations and iterations. Finally, we showed that our descent method can be combined with the subdivision algorithm from \cite{DSH2005} to compute approximations of entire Pareto sets. 
	
	For future work, we believe that it is straightforward to extend our method to constrained MOPs by adding constraints to the problem \eqref{eq:approx_desc_dir} that ensure that the descent direction maintains the feasibility of the descent sequence (similar to \cite{GPD2017} for smooth problems). Additionally, in \cite{CQ2013}, the classical gradient sampling method for scalar nonsmooth optimization was generalized by allowing variable norms in the direction finding problem, increasing its efficiency. We expect that a similar generalization can be performed for problem \eqref{eq:approx_desc_dir}, which potentially yields a similar increase in efficiency. Additional potential for increased performance lies in more advanced step length schemes as well as descent directions with memory (for instance, conjugate-gradient-like). Furthermore, it might be possible to extend our method to infinite-dimensional nonsmooth MOPs \cite{M2008,CM2020}. Finally, in the context of nonsmooth many-objective optimization, we believe that considering subsets of objectives is a very promising and efficient approach (cf.~\cite{GPD2019} for smooth problems). However, theoretical advances are required for locally Lipschitz continuous problems.
		
	\section*{Acknowledgement}
	This research was funded by the DFG Priority Programme 1962 ''Non-smooth and Complement-arity-based Distributed Parameter Systems''.		
		
	\bibliography{literature}   % name your BibTeX data base	

\begin{thebibliography}{10}

\bibitem{AGG2015}
H.~Attouch, G.~Garrigos, and X.~Goudou.
\newblock A dynamic gradient approach to pareto optimization with nonsmooth
  convex objective functions.
\newblock {\em Journal of Mathematical Analysis and Applications}, 422(1):741
  -- 771, 2015.

\bibitem{B2013}
Y.~Bello-Cruz.
\newblock {A Subgradient Method for Vector Optimization Problems}.
\newblock {\em SIAM Journal on Optimization}, Nov 2013.

\bibitem{BLO2005}
J.~Burke, A.~Lewis, and M.~Overton.
\newblock {A Robust Gradient Sampling Algorithm for Nonsmooth, Nonconvex
  Optimization}.
\newblock {\em SIAM Journal on Optimization}, 15:751--779, Jan 2005.

\bibitem{CG1959}
W.~Cheney and A.~A. Goldstein.
\newblock Proximity maps for convex sets.
\newblock {\em Proceedings of the American Mathematical Society},
  10(3):448--450, 1959.

\bibitem{CM2020}
C.~Christof and G.~Müller.
\newblock {Multiobjective Optimal Control of a Non-Smooth Semilinear Elliptic
  Partial Differential Equation}.
\newblock {\em European Series in Applied and Industrial Mathematics (ESAIM) :
  Control, Optimisation and Calculus of Variations}, 2020.

\bibitem{C1983}
F.~Clarke.
\newblock {\em Optimization and Nonsmooth Analysis}.
\newblock Society for Industrial and Applied Mathematics, 1983.

\bibitem{NSFL2013}
J.~Cruz~Neto, G.~Silva, O.~Ferreira, and J.~Lopes.
\newblock A subgradient method for multiobjective optimization.
\newblock {\em Computational Optimization and Applications}, 54:461--472, Apr
  2013.

\bibitem{CQ2013}
F.~E. Curtis and X.~Que.
\newblock An adaptive gradient sampling algorithm for non-smooth optimization.
\newblock {\em Optimization Methods and Software}, 28(6):1302--1324, 2013.

\bibitem{D2001}
K.~Deb.
\newblock {\em Multi-objective optimization using evolutionary algorithms},
  volume~16.
\newblock John Wiley \& Sons, 2001.

\bibitem{DPAM2002}
K.~{Deb}, A.~{Pratap}, S.~{Agarwal}, and T.~{Meyarivan}.
\newblock {A fast and elitist multiobjective genetic algorithm: NSGA-II}.
\newblock {\em IEEE Transactions on Evolutionary Computation}, 6(2):182--197,
  Apr 2002.

\bibitem{DSH2005}
M.~Dellnitz, O.~Sch{\"u}tze, and T.~Hestermeyer.
\newblock Covering {P}areto {S}ets by {M}ultilevel {S}ubdivision {T}echniques.
\newblock {\em Journal of Optimization Theory and Applications},
  124(1):113--136, 2005.

\bibitem{FGS2008}
J.~Fliege, L.~Graña, and B.~F. Svaiter.
\newblock {Newton's Method for Multiobjective Optimization}.
\newblock {\em SIAM Journal on Optimization}, 20, Mar 2008.

\bibitem{FS2000}
J.~Fliege and B.~F. Svaiter.
\newblock Steepest descent methods for multicriteria optimization.
\newblock {\em Mathematical Methods of Operations Research}, 51(3):479--494,
  Aug 2000.

\bibitem{GPD2017}
B.~Gebken, S.~Peitz, and M.~Dellnitz.
\newblock {A Descent Method for Equality and Inequality Constrained
  Multiobjective Optimization Problems}.
\newblock In L.~Trujillo, O.~Sch{\"{u}}tze, Y.~Maldonado, and P.~Valle,
  editors, {\em Numerical and Evolutionary Optimization – NEO 2017}, pages
  29--61. Springer, Cham, 2019.

\bibitem{GPD2019}
B.~Gebken, S.~Peitz, and M.~Dellnitz.
\newblock {On the hierarchical structure of Pareto critical sets}.
\newblock {\em Journal of Global Optimization}, 73(4):891--913, 2019.

\bibitem{G1977}
A.~Goldstein.
\newblock {Optimization of Lipschitz continuous functions}.
\newblock {\em Mathematical Programming}, 13:14--22, 12 1977.

\bibitem{HSS2016}
E.~S. Helou, S.~A. Santos, and L.~E.~A. Simões.
\newblock On the differentiability check in gradient sampling methods.
\newblock {\em Optimization Methods and Software}, 31(5):983--1007, 2016.

\bibitem{K1985b}
K.~C. Kiwiel.
\newblock A descent method for nonsmooth convex multiobjective minimization.
\newblock {\em Large scale systems}, 8(2):119--129, 1985.

\bibitem{K1985a}
K.~C. Kiwiel.
\newblock {\em Methods of {D}escent for {N}ondifferentiable {O}ptimization}.
\newblock Springer-Verlag Berlin Heidelberg, 1985.

\bibitem{K1990}
K.~C. Kiwiel.
\newblock {Proximity Control in Bundle Methods for Convex Nondifferentiable
  Minimization}.
\newblock {\em Mathematical Programming}, 46:105--122, Jan 1990.

\bibitem{K2010}
K.~C. Kiwiel.
\newblock {A Nonderivative Version of the Gradient Sampling Algorithm for
  Nonsmooth Nonconvex Optimization}.
\newblock {\em SIAM Journal on Optimization}, 20(4):1983--1994, 2010.

\bibitem{MY2012}
N.~Mahdavi-Amiri and R.~Yousefpour.
\newblock An {E}ffective {N}onsmooth {O}ptimization {A}lgorithm for {L}ocally
  {L}ipschitz {F}unctions.
\newblock {\em Journal of Optimization Theory and Applications},
  155(1):180--195, Oct 2012.

\bibitem{M2003}
M.~M. M{\"a}kel{\"a}.
\newblock Multiobjective proximal bundle method for nonconvex nonsmooth
  optimization: Fortran subroutine mpbngc 2.0.
\newblock {\em Reports of the Department of Mathematical Information
  Technology, Series B. Scientific Computing, B}, 13:2003, 2003.

\bibitem{MEK2014}
M.~M. M{\"a}kel{\"a}, V.-P. Eronen, and N.~Karmitsa.
\newblock {\em On Nonsmooth Multiobjective Optimality Conditions with
  Generalized Convexities}, pages 333--357.
\newblock Springer New York, New York, NY, 2014.

\bibitem{MKW2014}
M.~M. M{\"a}kel{\"a}, N.~Karmitsa, and O.~Wilppu.
\newblock Multiobjective proximal bundle method for nonsmooth optimization.
\newblock {\em TUCS technical report No 1120, Turku Centre for Computer
  Science, Turku}, 2014.

\bibitem{MN1992}
M.~M. M{\"a}kel{\"a} and P.~Neittaanm{\"a}ki.
\newblock {\em Nonsmooth Optimization: Analysis and Algorithms with
  Applications to Optimal Control}.
\newblock World Scientific, 1992.

\bibitem{M1998}
K.~Miettinen.
\newblock {\em Nonlinear {M}ultiobjective {O}ptimization}.
\newblock Springer US, 1998.

\bibitem{MKM2018}
O.~Montonen, N.~Karmitsa, and M.~M. M{\"a}kel{\"a}.
\newblock Multiple subgradient descent bundle method for convex nonsmooth
  multiobjective optimization.
\newblock {\em Optimization}, 67(1):139--158, 2018.

\bibitem{M2008}
B.~Mordukhovich.
\newblock Multiobjective optimization problems with equilibrium constraints.
\newblock {\em Optimization Methods and Software}, 117:331--354, Jul 2008.

\bibitem{SSW2002}
S.~Sch\"{a}ffler, R.~Schultz, and K.~Weinzierl.
\newblock {Stochastic Method for the Solution of Unconstrained Vector
  Optimization Problems}.
\newblock {\em J. Optim. Theory Appl.}, 114(1):209–222, 2002.

\bibitem{S1985}
N.~Shor.
\newblock {\em {Minimization Methods for Non-Differentiable Function}}.
\newblock Springer-Verlag Berlin Heidelberg, 1985.

\end{thebibliography}
	\bibliographystyle{abbrv}

\end{document}